\documentclass[11pt,reqno]{amsart}
\usepackage{amscd,amsmath,amsopn,amssymb,amsthm,multicol}
\usepackage{tikz,tikz-cd,anysize,verbatim,ifthen,xargs,colortbl,float,framed}
\usepackage{longtable,mathtools, hyperref, cleveref, tcolorbox}
\usepackage[english]{babel}
\usepackage[utf8]{inputenc}
\usepackage[numbers]{natbib}
\everymath=\expandafter{\the\everymath\displaystyle}

\usetikzlibrary{decorations.markings,arrows,arrows.meta,bending,calc}
\tikzset{>={Stealth[scale=1.5, bend]}}
\usetikzlibrary{shapes.geometric}

\usepackage{comment}
\usepackage{csquotes}
\usepackage{dsfont}
\usepackage{mathdots}

\usepackage{makecell} \setcellgapes{6pt}

\theoremstyle{plain}
\newtheorem{theorem}{Theorem}
\newtheorem*{theorem*}{Theorem}
\newtheorem{corollary}{Corollary}[theorem]
\newtheorem{prop}{Proposition}
\newtheorem{lemma}{Lemma}

\theoremstyle{definition}
\newtheorem{definition}{Definition}
\newtheorem{example}{Example}
\newtheorem{remark}{Remark}

\newcommand\com[1]{}

\newcommand\N{{\mathbb N}}

\newcommand\R{{\mathbb R}}

\makeatletter
\newcommand*\circled[2][1.6]{\tikz[baseline=(char.base)]{
    \node[shape=circle, draw, inner sep=1pt, 
        minimum height={\f@size*#1},] (char) {\vphantom{WAH1g}#2};}}
\makeatother

\makeatletter
\newcommand{\dashedcircle}[2][1.6]{%
  \tikz[baseline=(char.base)]{%
    \node[shape=circle, draw, dashed, inner sep=1pt,
        minimum height={\f@size*#1 pt},] (char) {\vphantom{WAH1g}$#2$};%
  }%
}
\makeatother

\makeatletter
\newcommand{\dashedbox}[1]{%
  \tikz[baseline=(char.base)]{%
    \node[draw, dashed, inner sep=1.9pt,
        minimum height={\f@size*1.6 pt},] (char) {\vphantom{WAH1g}$#1$};%
  }%
}
\makeatother

\makeatletter
    \def\@@and{}
\makeatother

\title[Jet-density of finite gap solutions]{Jet-density of finite-gap solutions for classes of BKM systems}

\author{Manuel Quaschner}
\address{Institute of Mathematics, Friedrich-Schiller-Universität, 07737 Jena, Germany.\newline
\ E-mail: {\tt manuel.robert.quaschner@uni-jena.de}.}

\author{Wijnand Steneker}
\address{Department of Mathematics and Statistics, UiT the Arctic University of Norway, Troms\o\ 9037, Norway.
\ E-mail: {\tt wijnand.s.steneker@uit.no}. }

\begin{document}

\begin{abstract}
We show that jets of initial data can be approximated up to arbitrary 
order by finite-gap solutions for classes of so-called BKM systems of PDEs introduced by Bolsinov--Konyaev--Matveev, 
which include classical PDEs such as KdV, Kaup--Boussinesq and Camassa--Holm. Finite-gap solutions are obtained via a finite-reduction map, defined algebraically, which sends solutions of a Stäckel system to solutions of the BKM PDE. For the classes containing KdV and Kaup--Boussinesq we obtain full jet-surjectivity via a triangular structure, whereas for the class containing Camassa--Holm we establish jet-surjectivity on an open set of initial data over $\mathbb{R}$ and a Zariski-open (dense) set over $\mathbb{C}$.
\end{abstract}

\maketitle

\tableofcontents

\section{Introduction}

The study of integrable partial differential equations (PDEs) has a long 
history in mathematics. There are many different notions of what integrability of a PDE should mean: 
\begin{itemize}
    \item the PDE admits a hierarchy of commuting flows (symmetries on infinite jet bundle);
    \item the PDE arises as the compatibility condition of a Lax pair;
    \item the existence of \enquote{sufficiently many} solutions that one can construct explicitly. 
\end{itemize}

There are many famous examples of integrable PDEs such as the Korteweg--
de Vries equation (KdV), Camassa--Holm equation (CH), Harry--Dym equation (HD) and many others. The construction of explicit solutions for the KdV equation has been extensively studied, and reveals the intimate interplay between the several notions of integrability. In their seminal paper on the KdV equation \cite{korteweg1895xli} in 1895, Korteweg and de Vries wrote down a solitary wave solution. During the 1970s a lot of progress was made on constructing explicit solutions for KdV by using its Lax formulation \cite{dubrovin1975inverse},\cite{its1975hill} ,\cite{marchenko1974periodic}, \cite{mckean1975spectrum}, \cite{novikov1974periodic}
and isospectrality of its Lax operator. The Lax operator for KdV is the Schrödinger--Hill operator
\begin{equation}
    - \partial_x^2 + q(x) 
\end{equation}
and for periodic potentials $q(x)$ it turns out that the spectrum consists of bands (closed intervals) separated by gaps. For the special class of finite-gap potentials $q(x)$, the spectrum consists of finitely many bands:
\begin{equation}
    [\lambda_1, \lambda_2] \cup [\lambda_3, \lambda_4] \cup \cdots \cup [\lambda_{2g-1}, \lambda_{2g}] \cup [\lambda_{2g+1}, \infty).
\end{equation}
One can build solutions from a collection of $g$ points, with one Dirichlet eigenvalue in each gap, and these give an alternative characterization of finite-gap solutions. Another method of obtaining finite-gap solutions is by considering stationary solutions of the hierarchy. This amounts to truncating the hierarchy at some order which produces an ODE inside the PDE. In the years following, the finite-gap solutions were also studied from the algebro-geometric point of view, where one associates to the band-edges $\lambda_1, \dots, \lambda_{2g+1}$ a Riemann surface of genus $g$ defined by
\begin{equation}
    \mu^2 = \prod_{i= 1}^{2g + 1} (\lambda - \lambda_i).
\end{equation}
Explicit quasi-periodic finite-gap solutions can then be expressed in terms of theta functions associated to the Jacobi variety \cite{dubrovin1976non}, \cite{krichever1977methods}, \cite{bernatska2024reality}.

In the recent paper \cite{bolsinov2022applications} Bolsinov--Konyaev--Matveev constructed a general family of evolutionary PDEs containing many famous examples and showed the existence of a hierarchy of commuting symmetries and conservation laws for these. The details of the construction are delegated to \cref{sec:Background}, but we mention here that the family of BKM PDEs is parametrized by a triple $(n, L, m(\mu))$, where $n$ is a positive integer (number of components), $L$ a Nijenhuis operator in companion form and $m(\mu)$ a nonzero polynomial. For example, KdV comes from $m \equiv 1$, Camassa--Holm from $m(\mu) = \mu$. We note that these BKM PDEs were studied in the context of the Nijenhuis Geometry Program.

In a further paper \cite{bolsinov2025finite}, Bolsinov--Konyaev--Matveev constructed finite gap solutions for these systems using solutions from a Stäckel system together with a finite-reduction map, which embeds the solutions of the Stäckel system into the solutions of the PDE. We recall the construction of finite-gap solutions in \cref{sec:Background}, these are essentially solutions of a Hamiltonian system constructed from the following data: a positive integer $N$, a polynomial $c(\mu)$ and the BKM data $(n, L, m(\mu))$. The embedding comes from a finite-reduction map $\mathcal{R}$, which is defined through algebraic conditions involving the polynomials $c, m$ as well as two Nijenhuis operators. Konyaev and Matveev also constructed Lax pairs for these systems in \cite{konyaev2025lax}. Thus, BKM systems are integrable in the sense of the first two definitions. In this paper, we address establishing integrability in the third sense for several classes of BKM systems.

As it stands, the third definition of integrability is quite vague and we would like 
to make more precise what we will actually study. The BKM PDEs are (formal) 
evolutionary equations in $1+1$ dimensions, so a solution depends on one spatial parameter and one time parameter. By virtue of the Cauchy--Kovalevskaya theorem, we know that in the analytic setting initial data $x \mapsto u(x,0)$ of an evolutionary PDE locally determines a solution $(x,t) \mapsto u(x,t)$. 

Therefore, we aim to approximate an arbitrary initial data by finite-gap solutions. There are different approaches to approximation. In this paper we 
focus on the problem of approximating the $k$-jets of the initial data up to 
any order. This is rather a formal approximation. We address the question of uniform convergence in the outlook. We emphasize that we study the spatial dynamics in $x$ of the finite-gap solutions. In \cite{bolsinov2025finite} it is explained how a quasi-linear 
PDE system defines dynamics in time $t$ on the space of solutions $u(x)$.

Our main theorems state that we can approximate the $k$-jet at $x = 0$ of the initial data up to any order for the case (i) $\deg(m) = 0$, $n \in \mathbb{N}$ and (ii) $\deg(m) = 1$ and $n = 1$. We analyze the structure of the Stäckel integrable system that determines the finite-gap solutions as well as the finite-reduction map (which maps solutions of Stäckel system to solutions of BKM PDE). We note that the idea of the proof could be applied to other PDEs with finite-gap solutions. In the BKM systems that we shall consider, we can control both the equations coming from the Stäckel system as well as the finite-reduction map. We obtain the following theorems. (For the notation and background on BKM systems and jets we refer the reader to Section~\ref{sec:Background}.)

\begin{theorem}[\textbf{Approximating jets of initial data for $\deg(m) = 0$}] \label{thm:ApproximationCasem=1}
	Consider an $n$-component BKM system with $\deg(m) = 0$. Let $\mathcal{R}: \mathbb{R}^N \rightarrow 
	\mathbb{R}^n$ be the finite-reduction map for the Stäckel system with $N$ gaps. A $k$-jet of any initial data $x \mapsto u(x,0)$ can be realized as the $k$-jet of a finite-gap solution. More precisely, for any $k$-jet $j_0^ku$ there exists $N \in \mathbb{N}$ and an initial condition $(w(0), p(0))$ determining a solution $w$ such that $\mathcal{R}(j_0^kw) = j^k_0u$.
\end{theorem}

Theorem~\ref{thm:ApproximationCasem=1} includes for example the well-known cases for the KdV equation and Kaup--Boussinesq. We note that similar results to Theorem~\ref{thm:ApproximationCasem=1} have been obtained by Blaszak--Szablikowski--Marciniak \cite{blaszak2023stackel}, \cite{szablikowski2024stationary}, using explicit Lax representations and stationarity of the hierarchy to obtain an embedding of an ODE into the PDE.

\begin{theorem}[\textbf{Approximating jets of initial data for $\deg(m) = 1, n=1$ on an open set}] \label{thm:ApproximationCasemlinear}
	Consider a $1$-component BKM system with $\deg(m) = 1$. Let $\mathcal{R}: \mathbb{R}^N \rightarrow \mathbb{R}$ be the finite-reduction map for the Stäckel system with $N$ gaps. For an open set of $k$-jets $j_0^ku$ there exists $N \in \mathbb{N}$ and an initial condition $(w(0), p(0))$ determining a solution $w$ such that $\mathcal{R}(j_0^kw) = j^k_0u$. 
    
    Moreover, the BKM system and finite-reduction map may be viewed over the complex numbers. Over $\mathbb{C}$, the statement holds true for a Zariski-open (in particular: dense) set of $k$-jets.
\end{theorem}

Theorem~\ref{thm:ApproximationCasemlinear} includes for example the case of the Camassa--Holm equation. Note that the theorem asserts jet-surjectivity for an open set of jets (over the reals). We believe that full jet-surjectivity is also true in our setting and comment on it after the proof. Numerical experiments suggest that the assumption $n = 1$ is not essential, although we prove the result only in this case. 
\medskip

\textbf{Structure of the Article.} We recall the important definitions and statement in Section~\ref{sec:Background}. Starting with BKM systems, we consider the corresponding Stäckel systems and how to produce finite gap solutions from them using the reduction map $\mathcal{R}$. Then we recall some jet notation that we will use for the formulation of our results.

In Section~\ref{sec:FormulationProblem} we state the problem we are considering in a more formal way and derive some properties of the reduction map $\mathcal{R}$. We also give two examples that indicate our strategy for the proof of Theorem~\ref{thm:ApproximationCasem=1}.

Section~\ref{sec:GradingVariables} introduces our notation of grading of the Hamiltonians. It involves considering which terms could occur for the higher order Taylor coefficients and in the potentials for the Stäckel systems. We also investigate the potentials for different choices of BKM systems and derive some useful properties. 

Section~\ref{sec:Proofm=1} is devoted to the proof of Theorem~\ref{thm:ApproximationCasem=1}. This is a straightforward calculation using the concepts we introduced. \\
Section~\ref{sec:ProofDegm=1} then focuses on the proof of Theorem~\ref{thm:ApproximationCasemlinear}. This is split into two cases, namely $m(\mu) = \mu$ and a general polynomial of degree $1$. For the first case, we have to do some lengthy calculations for deriving how the coefficients of the Taylor approximation depend on our choice of initial conditions for the Stäckel system. We can prove that this dependence is nontrivial, hence the image is an open set. We will comment on what needs to be done to prove surjectivity of this map. The result for the general case can then be reduced to the Camassa--Holm case by a suitable gauge transformation.

Finally, in the outlook (Section~\ref{sec:Outlook}) we will comment on which other cases one could try to solve next and our thoughts on the problem of approximation in the other senses.

The appendices contain some explicit calculations from 
Section~\ref{sec:ProofDegm=1} and explicit examples for small dimension that were postponed in order not to interrupt the flow of reading.

\section{Background} \label{sec:Background}

\subsection{BKM systems of PDEs}

We recall the construction of the BKM family of systems of 
integrable PDEs. The family is a (formal) evolutionary PDE parametrized by a continuous 
parameter $\lambda \in \mathbb{R} \cup \{\infty\}$ and can be 
further classified into four types depending on (i) $\lambda 
\in \mathbb{R}$ or $\lambda = \infty$ and (ii) whether $
\lambda$ is a root (of the polynomial $m$ in the construction data) or not. The construction of the BKM PDEs depends on the following data:
\begin{itemize}
    \item the number $n$ of components (that is, the number of dependent variables)
    \item a Nijenhuis operator $L$ on $\mathbb{R}^n(u)$ that is \textit{differentially nondegenerate} (coefficients of characteristic polynomial are functionally independent a.e.).
    \item a polynomial $m(\mu) = m_n \mu^n + m_{n-1} \mu^{n-1} + \cdots + m_1 \mu + m_0 $ of degree $\leq n$. 
\end{itemize}
Let us comment on these conditions. There are $n$ functions $u^1(x,t), \dots, u^n(x,t)$ 
which appear in the PDE. Additionally, there will be a differential constraint equation in the form of an ODE (differentiation on $x$) on an auxiliary function $q(x,t)$. Recall that $L$ being a Nijenhuis operator means that its \textit{Nijenhuis torsion} vanishes
\begin{equation}
    \mathcal{N}_L(v,w) = L^2[v,w] + [Lv, Lw] - L[Lv,w] - L[v, Lw],
\end{equation}
where $[\cdot, \cdot ]$ denotes the commutator. The condition that $L$ in addition is nondegenerate means that the coefficients of the characteristic polynomial $\sigma_u(L)$ (subscript $u$ to emphasize the correspondence $u \leftrightarrow \sigma_u$)
\begin{equation}
    \sigma_u(L)(\mu) := \det(\mu\, \mathrm{Id} - L)
\end{equation}
are functionally independent, which implies that the coefficients can be taken as coordinates. In these coordinates, the Nijenhuis operator $L(u)$ depending on the variables $u = 
(u^1, \dots, u^n)$ takes the following companion matrix form:
\begin{align*}
	L &= \begin{pmatrix}
	-u_1 & 1 & 0 & \ldots & 0 \\
	-u_2 & 0 & 1 & \ddots & \vdots \\
	\vdots & & \ddots & \ddots & \\
	-u_{n-1} & & & 0 & 1 \\
	-u_{n} & 0 & & & 0
	\end{pmatrix}.
	\end{align*}
In this case, we say that $L$ is in \textit{companion form} with respect to the coordinates 
$u^i$. In the case that $m$ is not of maximal degree 
$n$ (i.e., $m_n = 0$), we say that it has a root at infinity. In order to define the BKM 
PDE, we introduce the vector field $\zeta_0$ which satisfies
\begin{equation}
    \mathcal{L}_{\zeta_0}\sigma_u = 1.
\end{equation}
As a consequence of differential nondegeneracy, this equation is well-defined. Next, we define the vector field $\zeta = m(L)\, \zeta_0$. We define the BKM system of PDEs by the equation
\begin{equation}
   \begin{cases}
       u_{t_{\lambda}} & = q_{xxx} \left(L - \lambda Id \right)^{-1}\zeta + q (L - \lambda Id)^{-1}u_x \\
       1  & = m(\lambda) (q_{xx}q - \frac12 q_x^2) + \sigma_u(\lambda)q^2.
   \end{cases}
\end{equation}
Here $t_{\lambda}$ is the time-coordinate but indexed with $\lambda$. We may rewrite this equation as a formal evolutionary flow by formally inverting the differential operator acting on $q$. We have the following four BKM types.

\textbf{BKM I}. If $\lambda \in \mathbb{R}$ and $\lambda$ is not a root of $m(\mu)$ then it is 
of type I:
\begin{equation}\label{eq:BKM_typeI}
    \begin{split}
        u_{t_{\lambda}} & = q_{xxx} \left(L -\lambda \text{Id} \right)^{-1} \zeta + q \left(L - \lambda \text{Id} \right)^{-1} u_x \\
        1 & = m(\lambda ) (q_{xx}q - \frac{1}{2} q_x^2) + \sigma_u(\lambda) q^2
    \end{split}
\end{equation}

\textbf{BKM II}. If $\lambda \in \mathbb{R}$ is a root $m(\lambda) = 0$ the constraint equation reduces to $q = \frac{1}{\sqrt{\sigma_u}}$. We call this case BKM type II:
\begin{equation}\label{eq:BKM_typeII}
    \begin{split}
        u_{t_{\lambda}} & = \left(\frac{1}{\sqrt{\sigma_u(\lambda)}}\right)_{xxx} \left(L -\lambda \text{Id} \right)^{-1} \zeta + \frac{1}{\sqrt{\sigma_u(\lambda)}} \left(L - \lambda \text{Id} \right)^{-1} u_x
    \end{split}
\end{equation}

\textbf{BKM III}. If $\lambda = \infty$ and $m_n \neq 0$, we obtain the following 
PDE, BKM type III:
\begin{equation}\label{eq:BKM_typeIII}
    \begin{split}
        u_{t_{\infty}} & = q_{xxx} \zeta + (L + q\ \text{Id}) u_x \\
        0 & = 2q + m_n q_{xx} - \text{tr}(L), 
    \end{split}
\end{equation}
where $m_n$ is the highest-degree coefficient of the polynomial $m(\mu)$. 

\textbf{BKM IV}. If $\lambda = \infty$ and $m_n = 0$ (i.e., root at infinity), then the 
constraint equation becomes $q = \frac{1}{2}\text{tr}(L)$, and so we obtain 
\begin{equation}\label{eq:BKM_typeIV}
    u_{t_{\infty}} = \frac{1}{2} (\text{tr}(L))_{xxx} \zeta + (L + \frac12 \text{tr}(L)\cdot \text{Id})\ u_x
\end{equation}

\subsection{Stäckel integrable systems}

A BKM system is constructed from the data $(n, L(u), m(\mu))$. To each BKM PDE we can 
associate a family of finite-dimensional Stäckel integrable systems indexed by a positive 
integer $N$ called the \textit{gap number}. The purpose of these integrable 
systems is to construct special solutions of the PDE, the so-called finite gap 
solutions. In the literature these finite-dimensional ODEs are often obtained by 
considering a stationary flow of the PDE hierarchy \cite{blaszak2023stackel}. Here we 
shall take a slightly different approach by first defining the Stäckel integrable system 
and then defining an embedding into the configuration space $\mathbb{R}^n(u)$ of the BKM 
PDE. The following is based on \cite[Section 1.3]{bolsinov2025finite} and the references therein.

We shall construct the \textit{reduced BKM system} for a BKM system with triple $(n, L, m)$
from the following additional data:
\begin{itemize}
    \item The number of gaps $N \in \mathbb{N}$, which can be arbitrarily large.
    \item A monic polynomial $c(\mu)$ of degree $2N + n$ of the form
    \begin{equation}
        c(\mu) = \mu^{2N+n} + \sum_{i=2}^{2N+n} c_i \mu^{2N+n-i}.
    \end{equation}
\end{itemize}
We now construct the $N$ Poisson-commuting Hamiltonians on $T^{*}\mathbb{R}^N$, equipped 
with the standard symplectic structure $\omega = \sum_i dw_i \wedge dp_i$ with respect to the cotangent 
coordinates $(w,p)$. We define a Nijenhuis operator $M$ and an inverse metric $\mathrm{g}^{-1}$ by declaring the $w^i$ to be companion coordinates
\begin{equation}\label{eq:M_and_g}
    M = \begin{pmatrix}
	-w_1 & 1 & 0 & \ldots & 0 \\
	-w_2 & 0 & 1 & \ddots & \vdots \\
	\vdots & & \ddots & \ddots & \\
	-w_{N-1} & & & 0 & 1 \\
	-w_{N} & 0 & & & 0
	\end{pmatrix}\hspace{1cm} \text{g}^{-1} = \begin{pmatrix}
	0 & \ldots & 0 & 0  & 1 \\
	0 & \ldots & 0 & 1  & w_1 \\
	\vdots & \iddots  & \iddots & \iddots  & w_2 \\
	0 & 1 &  w_1 & \iddots &  \vdots \\
	1 & w_1 & w_2 & \ldots & w_{N-1}.
	\end{pmatrix}
\end{equation}
The metric $\mathrm{g}$ is flat and the pair $(\mathrm{g}, M)$ is geodesically compatible in the sense of \cite{bolsinov2024applications}. 
We construct the kinetic energies from the inverse metric $g^{-1}$ and 
the Nijenhuis operator $M$:
	\begin{equation}\label{def_kinetics}
	\begin{split}
	& \frac{1}{2} g^{-1}(\det(\mu \cdot Id_N - M)(M^{T} - \mu Id_N )^{-1} p,p)\\
	& = K_1(w, p)\ \mu^{N-1} + K_2(w,p) \mu^{N-2} + \dots + K_{N}(w,p).
	\end{split}
	\end{equation}
Note that the kinetic energies do not depend on the choice of polynomial $c(\mu)$. Moreover, the kinetic energy $K_1$ is given by the flat 
(inverse) metric.

We now construct the potential energies. In order to couple the dynamics of the reduced system 
to the original BKM system, we define the rational function $f(\mu) := \frac{c(\mu)}{m(\mu)}$. 
We apply the function $f$ to the Nijenhuis matrix $M$ to obtain a matrix $f(M)$. Since this 
matrix commutes with $M$, we can express it as a linear combination of powers of $M$. As the matrix is gl-regular, we need all of the matrices $M^{N-1}, \ldots, M, \text{Id}$ and can write
\begin{equation}\label{def_potentials}
\begin{split}
	f(M) = V_1(w)\ M^{N-1} + \dots + V_{N}(w)\ \text{Id}_N
\end{split}
\end{equation}
The latter equation defines the potential energies $V_i(w)$. The $N$ Hamiltonian functions
\begin{equation}
    H_i := K_i(w,p) + V_i(w),\ i = 1, \dots, N \label{eq:Hamiltonians}
\end{equation}
on $T^{*}\mathbb{R}^N(w,p)$ commute with respect to the canonical Poisson bracket and 
are functionally independent almost everywhere, so that $F := (H_1, \dots, H_N)$ defines an 
integrable system on $T^{*}\mathbb{R}^N$.

Moreover, the system is of Stäckel type, so that the Hamilton--Jacobi equations can be separated completely. If we take the eigenvalues of $M$ to be coordinates $q^i$ and do the corresponding coordinate change using symmetric polynomials, the system takes a different form. One useful property of the new system is
\begin{equation}
    p_i^2 = \frac{c(q^i)}{m(q^i)} =: f(q^i)
\end{equation}
on the common zero level set.

\begin{remark}
The notion of Stäckel integrability is closely related to (additive) separation of variables. A Stäckel integrable system is one for which the $N$ Hamilton--Jacobi equations can be separated by a common generating function $S$.
\end{remark}

\subsection{Finite-reduction map and producing finite-gap solutions}\label{subsect_fin_red}

We now describe the finite-reduction procedure for a BKM system constructed 
from the data $(n,L(u),m(\mu))$ following \cite{bolsinov2025finite}. In the previous 
section, we constructed from the additional data of a gap number $N$, a monic polynomial 
$c(\mu)$ of degree $2N + n$, and a Nijenhuis operator $M(w)$ (companion form) a Stäckel 
integrable system on $T^{*}\mathbb{R}^{N}$ which is coupled to the BKM PDE 
through a function $\frac{c(\mu)}{m(\mu)}$ that is involved in generating the 
Hamiltonians. The procedure will be described through a reduction map $
\mathcal{R}: \R^N \rightarrow \R^n$ which sends a solution $w(x,t)$ of the 
joint flow, given in companion coordinates, to a solution $(u_1(x,t), \dots, 
u_n(x,t))$ of our BKM system in consideration. It can be defined via 
algebraic relations (assuming $L$ in companion form).

We give an implicit definition of the map $\mathcal{R}$. Denote by $\sigma_u(\mu) = \det(\mu\, \mathrm{Id} - L(u))$ and $\sigma_w(\mu) =  \det(\mu\, \mathrm{Id} - M(w))$ the characteristic polynomials of $L, M$, respectively. We 
consider the following polynomial which is of degree at most $2N+n-1$ (the highest order term cancels):
\begin{equation}
    \sigma_u(\mu) \cdot \sigma_w(\mu)^2 - c(\mu).
\end{equation}
In order to define $\mathcal{R}$ we impose the following two conditions:
\begin{itemize}
    \item  The polynomial must be divisible by $m(\mu)$. 
    \item The resulting polynomial (after division) is of degree at most $2N-1$.
\end{itemize}
The above conditions give us a set of $n$ equations involving $u$ and $w$. The map $\mathcal{R}$ is then defined by resolving 
these equations for $u$ in terms of $w$.

\begin{remark}
In the paper \cite{bolsinov2025finite}, the authors constructed the finite-reduction 
map by considering two equations. The \textit{base equation} is a 2nd order ODE on the coefficients $w_i(x)$ of characteristic polynomial $\sigma_w$. The base equation plays a role analogous to a Sturm--Liouville operator; for KdV it is equivalent to the Schrödinger--Hill operator which is a genuine Sturm--Liouville operator. A solution $w(x)$ of the base equation can be extended to a solution $w(x,t)$ through a quasi-linear PDE system of hydrodynamic type, that is, $w_t = A(w) w_x$. The solution $w(x,t)$ is a solution of the joint flow of the Stäckel system, and determines a solution $u(x,t) = \mathcal{R}(w(x,t))$.
\end{remark}

We can interpret the condition in terms of the roots of the polynomial $m$. If $\lambda$ 
is a simple root of $m$, the condition is that $\sigma_u(\lambda) \cdot 
\sigma_w(\lambda)^2 - c(\lambda) = 0$. For a multiple root of multiplicity $l$, the 
corresponding derivatives (up to order $(l-1)$) with respect to $\mu$ of the polynomial $
\sigma_u(\mu) \cdot \sigma_w(\mu)^2 - c(\mu)$ also have to vanish. All of this also 
applies to roots at infinity as well. In 
order to handle the equations we change variables $\mu = \frac{1}{\lambda}$ and multiply 
the resulting equation by $\lambda^{2N+n-2}$ and then consider the root $\lambda = 0$. Thus, a root at infinity will produce equations on the highest degree terms (the number depending on multiplicity).

We can use the finite-reduction map to produce finite gap solutions for the BKM system 
following \cite{bolsinov2025finite}. There is an analogous statement for BKM I and II 
(only constraint function takes a different form), but it is not required for our 
purposes. 

\begin{theorem*}[\textbf{Producing Finite Gap Solutions}\cite{bolsinov2025finite}]
Consider an $n$-component BKM system of type (III) or (IV) constructed from the
polynomial $m(\mu)$. Let $N \geq 1$ and $c(\mu)$ be the monic polynomial
\begin{equation}
    c(\mu) = \mu^{2N+n} + \sum_{i= 2}^{2N + n} c_{i}\ \mu^{2N + n - i }.
\end{equation}
Denote by $F = (H_1, \dots, H_N): T^{*}\mathbb{R}^N \rightarrow \mathbb{R}^N$ 
the integrable Stäckel system determined by $(N, c(\mu), m(\mu), n)$. 
Consider the joint flow $(x,t) \mapsto \phi_1^x \circ \phi^t_2$ of the two 
commuting Hamiltonians $H_1, H_2$, and let $X \subseteq T^{*}\mathbb{R}^N$ be 
the level set defined by the vanishing of all the Hamiltonians 
\begin{equation}\label{def_common_level_set}
    X  := F^{-1}(0) = \{H_1 = 0,\ \dots,\ H_N = 0 \}.
\end{equation}
If $w(x,t)$ is a solution of the joint flow on the level $X$, then
\begin{equation}\label{def_finite_gap_solution}
    \begin{split}
        u(x,t ) & = \mathcal{R}(w(x,t)) \\
        q(x,t) & = w_1(x,t).
    \end{split}
\end{equation}
defines a solution of the BKM system. (NB: Here the constraint $q(x,t)$ is only relevant for BKM III.) We call $u = \mathcal{R}(w)$ a \textup{finite-gap solution}. \\
For the case of a BKM system of type I or II, the time flow is generated by a 
Hamiltonian which is a linear combination of the Hamiltonians $H_1, \dots, H_N$ (see \cite[~Thm.2.2.]{bolsinov2025finite}). If $w(x,t)$ is a solution of the joint flow 
on the level $X$ and $a = \sqrt{c(\lambda)}$, then
\begin{equation}
    \begin{split}\label{def_finite_gap_solution_I_and_II}
        u(x,t) & = \mathcal{R}(w(x,at)) \\
        q(x,t) & = \frac{1}{a} \sigma_{w(x, at)}(\lambda)
    \end{split}
\end{equation}
defines a solution of the BKM system. (NB: Here the constraint $q(x,t)$ is only relevant for BKM I.) We call $u = \mathcal{R}(w)$ a \textup{finite-gap solution}.
\end{theorem*}

The previous theorem requires working on the common level set of the 
Hamiltonians. For our purposes, it will be convenient to work on the entire 
cotangent bundle $T^{*}\mathbb{R}^N$ as this means we have more free parameters 
(the initial conditions $w(0),p(0)$) at our disposal to approximate the 
initial data. To this end, suppose that the polynomial $c(\mu)$ is of the 
following special form
\begin{equation}\label{eq_special_c_polynomial}
    c(\mu) := \mu^{2N+n} + m(\mu) d(\mu) +r(\mu) = \mu^{2N+n} + m(\mu) \sum_{i=1}^{N} d_i \mu^{N-i} +r(\mu),
\end{equation}
i.e. $d(\mu)$ is an arbitrary polynomial of degree at most $N-1$ and $r(\mu)$ is a polynomial of degree strictly less than $\deg(m)$. The kinetic 
energy does not depend on the constants $d$, but the potentials do depend on 
the constants $d$ via the function $f(\mu) = \frac{c(\mu)}{m(\mu)}$ determining 
the potential cf. \eqref{def_potentials}. Now, let $\tilde{f}(\mu) := 
\frac{\mu^{2N+n}}{m(\mu)}$, $\tilde{r}(\mu):= \frac{r(\mu)}{m(\mu)}$ and denote by $\tilde{H}_i$ the Hamiltonians with 
respect to the function $\tilde{f} +\tilde{r}$. Then we obtain that
\begin{equation}\label{eq_shifted_hamiltonian}
    H_i = \tilde{H}_i + d_i.
\end{equation}
Moreover, the constants $d_1, \dots, d_N$ do not appear in the Hamiltonian 
equations corresponding to both $H_i$ and $\tilde{H}_i$. In particular, the $\tilde{H}_i$ are integrals of motion for $F = (H_1, \dots, H_N)$. Thus, if we have an arbitrary initial condition $w(0), p(0)$ and set $d_i := - \tilde{H}_i(w(0), p(0))$, the solution $(w(x,t), p(x,t))$ for this initial condition will lie on the common level set $X$ \eqref{def_common_level_set}. We formulate this as a corollary.

\begin{corollary}\label{cor:fin_gap_entire_cotangent_bundle}
Suppose we are in the setting of the above theorem. For any initial condition $(w(0), p(0)) \in T^{*}\mathbb{R}^N$, consider the polynomial $c(\mu)$ as in \eqref{eq_special_c_polynomial}, where we set $d_i := -\tilde{H}_i(w(0), p(0))$ for $i = 1, \dots, N$. Given any initial condition, let $w(x,t)$ be the solution of the joint Hamiltonian flow generated by $H_1$ and the auxiliary Hamiltonian. Then the pair of functions $u(x,t),\ q(x,t)$ as in \eqref{def_finite_gap_solution} or \eqref{def_finite_gap_solution_I_and_II} defines a finite-gap solution of the BKM system. (The function $q(x,t)$ is only 
relevant for BKM type I or III.)
\end{corollary}

\begin{remark}
Equivalently, one may consider the constants $d_i$ as Casimirs on $T^{*}\mathbb{R}^N 
\times \mathbb{R}^N(d_1, \dots, d_N)$ and consider the leaf that passes through the 
chosen initial condition.    
\end{remark}

\subsection{Jets of curves}

We briefly recall the notion of jets. We shall use this in the next section to formulate the problem of jet approximation more precisely. Let $f: I \rightarrow \mathbb{R}^{m}$ be a smooth curve. The $k$-jet of $f$ at $x = a$ consists of all its derivatives at $a$ up to order $k$, and is denoted by $j_a^kf$:
\begin{equation}
    j^k_af = (f(a), f'(a), f''(a), \dots, f^{(k)}(a)),
\end{equation}
which may be identified with the $k$th order Taylor expansion of the curve at $a$. The space of all $k$-jets of curves $\mathbb{R} \rightarrow \mathbb{R}^m$ 
forms the $k$th-order jet bundle $J^k(\mathbb{R}, \mathbb{R}^m)$. Given coordinates $(y^i)$ on $\mathbb{R}^m$, we 
obtain coordinates $(x, y^i_j)$ on $J^k(\mathbb{R}, \mathbb{R}^m)$ where $y^i_j$ agrees with the derivative $\frac{d^j y^i}{dx^j}$ when evaluated on a curve $x \mapsto y(x)$. We define the total derivative $D_x$ as the vector field 
\begin{equation}
    D_x = \frac{\partial}{\partial x} + \sum_{i,j} y^i_{j+1} \frac{\partial }{\partial y^i_{j}},
\end{equation}
which has the property that $D_x(u_j) = u_{j+1}$ for all $j$.

The \textit{$k$th jet-prolongation} of $f$ is the section $I \rightarrow J^k$ defined by $x 
\mapsto j_x^kf$. Similarly, given a smooth map $F: M \rightarrow N$ between manifolds $M,N$ 
we get induced maps $J^k(\mathbb{R}, M) \rightarrow J^k(\mathbb{R}, N)$ defined through $j^kF: j^k_af \mapsto j^k_a(F \circ f)$.

The space of infinite jets $J^{\infty}$ is an infinite-dimensional manifold defined as the 
inverse limit of $J^k$ under the projections $\pi_{k, k-1}$, each $\infty$-jet consists of 
all derivatives of a curve at a point. However, the analogy between $\infty$-jets and 
Taylor expansions breaks down, because not every $\infty$-jet gives rise to a convergent 
Taylor expansion.

An ODE of order $k$ is a submanifold $\mathcal{E} \subseteq J^k$. Locally, the ODE is given 
by the vanishing of a jet-function $F(x, y^i_{j})$. We define its \textit{first 
prolongation} $\mathcal{E}^{(1)} \subseteq J^{k+1}$ to be the ODE of order $k+1$ given by 
\begin{equation}
    \mathcal{E}^{(1)} := \{F(x,u^i_j) = 0, (D_xF)(x,u_j^i) = 0 \},
\end{equation}
i.e., the first prolongation is obtained by adding to the original ODE the differentiated 
equations. Higher prolongations are defined inductively via $\mathcal{E}^{(l)} := 
(\mathcal{E}^{(l-1)})^{(1)}$. A point $z = j^k_af \in \mathcal{E}^{(l)}$ may be viewed as a solution up to order $k + l$ at $x = a$.

\section{Formulation of Problem}
\label{sec:FormulationProblem}

\subsection{Triangular form of finite-reduction map for $\deg(m) = 0$ (e.g. KdV, KB)}

For the case with polynomial $\deg(m) = 0$, we show that the finite 
reduction map has a triangular structure in terms of the companion 
coordinates $w^i$. An important consequence of this is that 
approximating $(u_1, \dots, u_n)$ is equivalent to approximating $(w_1, 
\dots, w_n)$. As we 
increase the gap number $N \rightarrow \infty$, we obtain more initial 
conditions and can approximate $(w_1, \dots, w_n)$ up to increasingly 
greater order. We will make this more precise in subsection 
\cref{subsec:formulation_approximating}.

\begin{prop}[\textbf{Form of the reduction map}] \label{prop:ReductionMap}
Consider a BKM PDE with $\deg(m) = 0$ and consisting of $n$ 
components. For a finite-gap solution $u = \mathcal{R}(w)$ we have a triangular 
structure in the sense that $u_i$ depends on $w_1, \dots, w_i$. Additionally, each 
$u_i$ depends linearly on $w_i$.
\end{prop}
\begin{proof}
For the definition of the map $\mathcal{R}$ we refer to Section~\ref{subsect_fin_red}. 
As $m$ is a nonzero constant, we only have roots at infinity, and hence the condition 
on the divisibility by $m$ is always satisfied. Thus, we have to consider the
degree conditions for $\sigma_u(\mu) \cdot \sigma_w(\mu)^2 - c(\mu)$. This could be a 
polynomial of degree $2N+n-1$ (note that the polynomials are monic and hence the terms of 
degree $2N+n$ cancel) and the condition is that all coefficients of order $\geq 2N$ vanish. 
Defining $w_0 :=1, u_0 := 1$, we obtain the following compact formula for $\sigma_u(\mu) \cdot \sigma_w(\mu)^2$:
\begin{equation}
    \sigma_u(\mu) \cdot \sigma_w(\mu)^2 = \sum_{i=0}^{2N+n} \left(\sum_{j+k+l=i} u_j w_k w_l \right) \mu^{2N+n-i},
\end{equation}
where the inner sum runs over all triples $(j,k,l)$ satisfying $j + k + l = i$ with $0 \leq 
j \leq n,\ 0 \leq k,l \leq N$. So for each $i \in \{1, \ldots, n\}$ we must have $
\sum_{j + k + l = i} u_j w_l w_k 
= c_{i}$. \\
From the formula it is immediately apparent that only $u_1, \ldots, u_i$ are involved, where 
$u_i$ enters linearly (as $w_0=1$). Additionally, only $w_1, \ldots, w_n$ can appear. Via an 
inductive argument (starting from $i=1$) we see that each $u_i$ is expressed as a polynomial 
in $w_1, \ldots, w_n, c_1, \ldots, c_n$ where $w_i$ and $c_i$ appear linearly.
\end{proof}

\subsection{Approximating initial data of BKM by finite gap solutions} \label{subsec:formulation_approximating}

By the Cauchy--Kovalevskaya theorem, any analytic initial data $x \mapsto u(x)$
locally determines a solution $u(x,t)$ for the BKM PDE. A finite-gap solution $u = \mathcal{R}(w)$ is a solution of the BKM PDE. It is natural to consider whether finite-gap solutions are dense in the space of all solutions of the PDE (endowed with a suitable topology). Here we approach this problem from the viewpoint of jets, that is, whether we can obtain the Taylor expansion of $u(x)$ up to arbitrary order from finite-gap solutions. This leads us to the following definition.

\begin{definition}[\textbf{Approximation by finite-gap solutions}]
We say that the initial data $u(x)$ can be \textit{approximated up to order $k$} (around $x = 0$) by a finite-gap solution if there exists a finite-gap solution $\mathcal{R}(w)$ with the same $k$-jet at $x = 0$.
\end{definition}

For the PDEs considered in the previous section, we note that 
the $k$-jet $j^k_0u$ of a finite-gap solution $u = \mathcal{R}(w)$ is determined by the $k$-jets $j^k_0w_{i_1}, \dots, j^k_0w_{i_s}$, where $s$ is a fixed integer independent of the gap number $N$. Thus the problem of approximating the initial data up to order $k$ reduces to the following problem, which is formulated solely in terms of the Hamiltonian flow.

\begin{oframed}
\textbf{Problem. (Approximating jets of curves by Hamiltonian flow.)}\\
Given an arbitrary $k$-jet $(j^k_0w_{i_1}, \dots, j^k_0w_{i_s})$, 
determine whether there exists a gap number $N$ and an initial condition 
$(w(0), p(0))$ such that the $k$-jet at $x = 0$ of the corresponding Hamiltonian
solution $x \mapsto (w_{i_1}(x), \dots, w_{i_s}(x))$ agrees with the prescribed $k$-jet.\end{oframed} 
\vspace{0.2cm}

We now reinterpret this problem in terms of the map on higher-order jet bundles induced by the Hamiltonian flow. To this end, consider the Hamiltonian $H_1$ (\ref{eq:Hamiltonians}) 
on 
$T^{*}\mathbb{R}^N = \mathbb{R}^{2N}$ and denote its flow by $\phi_{1}^x$. The Hamiltonian 
equations
\begin{equation}
    (w_i)_x = \frac{\partial H_1}{\partial p_i},\ (p_i)_x = -\frac{\partial H_1}{\partial w_i}
\end{equation}
define a first-order ODE $\mathcal{E} \subseteq J^1(\mathbb{R}, \mathbb{R}^{2N})$ in the 
first-order jet bundle. By differentiating the equations with respect to $x$ repeatedly, we 
obtain expressions for all derivatives of $w_i$ and $p_i$ up to order $k$. We collect all 
these equations to form the $(k-1)$-th prolongation of the Hamiltonian equations, which we 
denote by
\begin{equation}
    \mathcal{E}_{\text{Ham}}^{(k-1)} \subseteq J^k(\mathbb{R}, \mathbb{R}^{2N}).
\end{equation}
A point on $\mathcal{E}_{\text{Ham}}^{(k-1)} \subseteq J^k$ is given by the $k$-jet of a 
Hamiltonian solution. For our examples, the corresponding Taylor coefficients depend 
rationally on the initial conditions $(w(0), p(0))$. Thus, the Hamiltonian flow induces a smooth map
\begin{equation}
    j^k_{\text{Ham}}: \mathbb{R} \times \mathbb{R}^{2N} \rightarrow J^k(\mathbb{R}, \mathbb{R}^{2N}),\ (x_0, (w, p)) \mapsto j^k_{x_0}[x \mapsto \phi^{x-x_0}_1(w,p)]
\end{equation}
whose image is by definition $\mathcal{E}^{(k-1)}_{\text{Ham}}$, and assigns to 
an initial condition the $k$-jet of the corresponding Hamiltonian 
solution. If we fix a basepoint $x_0$, the fiber over $x_0$ in $J^k(\mathbb{R}, \mathbb{R}^{2N})$ has dimension $2N(k+1)$, while there are $2N$ initial conditions (and possibly 
additional parameters). In view of these dimensions, we can therefore only hope to approximate a 
fixed number of position coordinates up to order $k$. We introduce the following notion of $k$-jet surjectivity, which means that arbitrary $k$-jets of the selected position components $w_I :=(w_{i_1}, \dots, w_{i_s})$ are realized by the Hamiltonian flow.

\begin{definition}[\textbf{$k$-jet surjectivity and jet-density}]
A tuple of position components $w_{I} := (w_{i_1}, \dots, w_{i_s})$ satisfies $k$\textit{-jet surjectivity} with respect to the Hamiltonian flow if the map
\begin{equation}
    j^k\text{pr}_{I} \circ j^{k}_{\mathrm{Ham}} : \mathbb{R} \times
    \mathbb{R}^{2N} \longrightarrow J^k(\mathbb{R}, \mathbb{R}^s)
\end{equation}
is surjective, where $\text{pr}_{I} : \mathbb{R}^{2N}(w,p) \to \mathbb{R}^s(w_{i_1}, \dots, 
w_{i_s})$ denotes the canonical projection. The tuple $w_I$ satisfies \textit{jet-density} if it is $k$-jet surjective for all $k \in \mathbb{N}$.
\end{definition}

Therefore the problem is equivalent to establishing $k$-jet surjectivity for 
suitable position components $w_I$.

By translational invariance of the Hamiltonian flow along the $x$-parameter, it suffices to check surjectivity of $k$-jets over the basepoint $x = 0$. 

\subsection{Examples for KdV, Kaup--Boussinesq}\label{subsec:Examplesm=1}

We consider first two classical examples, namely the Korteweg--de Vries 
(KdV) equation and the Kaup--Boussinesq (KB) equation. These examples 
illuminate how to prove the result for all BKM systems with $\deg(m)=0$ up 
to any order.

Since we are working on the level set $X^N$, where all Hamiltonians vanish, we can pick 
$N$ initial 
conditions among $w_1(0), \dots, w_N(0), p_1(0), \dots, p_N(0)$. The other $N$ 
coordinates are then expressed through the Hamiltonians. Alternatively, following the 
approach in subsection \ref{subsect_fin_red} we can actually use $2N$ initial conditions by a suitable choice of $c(\mu)$. We take the latter approach in the 
proof of the general case, but for illustration purposes we work with $N$ initial 
conditions here. We also put the constants $c_2, \dots, c_N$ equal to zero.

\subsubsection{Example: KdV $(n=1)$  }

We have $n = 1$. The KdV equation is given by
\begin{equation}
    u_t = -\frac12 u_{xxx} + \frac32 u u_x
\end{equation}
with reduction map $\mathcal{R}: \mathbb{R}^N(w) \rightarrow \mathbb{R}(u)$ given by
\begin{equation}
    u = - 2 w_1.
\end{equation}
In particular, the form of the reduction map is independent of the polynomial $c$, but the polynomial $c$ does appear in the equations of motion. In order to 
approximate $u$ it suffices to approximate $w_1$. The higher-order coefficients of $w_1$ are determined through Hamiltonian equations, i.e., $j^kw_1 \in \mathcal{E}_{\text{Ham}}$. We compute the Taylor expansions of $w_1$ at $x = 0$ for some choices of $N$.
Let us start with $N = 4$:
\begin{equation}
    \begin{split} 
        w_1(x) &  = w_1(0) - p_4(0) x + \left(-\frac32 w_1(0)^2 + w_2(0)\right) x^2 
        + \left(\frac{2}{3} p_4(0) w_1(0) -\frac13 p_3(0)\right) x^3 + \mathcal{O}(x^4)
    \end{split}
\end{equation}
We see that we can approximate $w_1$ up to order $3$ by choosing the initial conditions $w_1(0), p_4(0), w_2(0), p_3(0)$, respectively. Note also that for each order the chosen initial condition does not appear in any lower-order terms. For example, when we look at the second-order expression $\left(-\frac32 w_1(0)^2  + w_2(0)\right)$, then $w_1(0)$ is expressed but we can always set it 
to any value by choosing $w_2(0)$ suitably.

Next, we look at $N = 5$:
\begin{equation}
    \begin{split}
        w_1(x) &  = w_1(0) - p_5(0) x + \left(-\frac32 w_1(0)^2  + w_2(0) \right) x^2 
        + \left(\frac{2}{3} p_5(0) w_1(0) -\frac13 p_4(0)\right) x^3 + \\
        & + \left(\frac{5}{6} w_1(0)^{3}  -\frac{5}{24} p_5(0)^{2}-\frac{5}{6}w_1(0) w_2(0) + \frac{1}{6} w_3(0) \right) x^4 +\mathcal{O}(x^5).
    \end{split}
\end{equation}
In this case, we can approximate $w_1$ up to order $4$ by choosing initial conditions $w_1(0), p_5(0), w_2(0), p_4(0), w_3(0)$. 
Finally, let us also look at $N = 6$:
\begin{equation}
    \begin{split}
        w_1(x) &  = w_1(0) - p_6(0) x + (\dots + w_2(0)) x^2 + (\dots - \frac{1}{3} p_5(0)) x^3 + (\dots + \frac16 w_3(0)) x^4 + (\dots - \frac{1}{30} p_4(0)) x^5 + \mathcal{O}(x^6)
    \end{split}
\end{equation}
Here we use dots to indicate that the expression depends on the coordinates appearing in lower-order coordinates. Similarly as above, we can approximate $w_1(x)$ up to order 5 by choosing $w_1(0), p_6(0), w_2(0), p_5(0), w_3(0),\\ p_4(0)$. 

\subsubsection{Example: Kaup--Boussinesq $(n = 2)$}

The KB equation is a 2-component system given by
\begin{equation}
    \begin{split}
        (u_1)_{t} & = (u_2)_x - \frac{3}{2} u_1 (u_1)_x  \\
        (u_2)_t & = -\frac12 (u_1)_{xxx} - u_2(u_1)_x - \frac12 u_1 (u_2)_x
    \end{split}
\end{equation}
The reduction map $\mathcal{R}: \mathbb{R}^N(w) \rightarrow \mathbb{R}^2(u)$ is 
given by
\begin{equation}
     u_1 = -2w_1,\ u_2 = 3w_1^2  - 2 w_2
\end{equation}
We see that approximating $u_1,u_2$ up to order $k$ is equivalent to approximating $w_1, w_2$ up to order $k$. 

Let us look at $N = 6$:
\begin{equation}
    \begin{split}
        w_1(x) & = w_1(0) + p_6(0) x + (\dots + w_3(0))\ x^2 + \mathcal{O}(x^3) \\
        w_2(x) & = w_2(0) + (-p_6(0) w_1(0) - p_5(0)) x + (\dots + w_4(0))\ x^2 + \mathcal{O}(x^3)
    \end{split}
\end{equation}
where dots denote terms involving variables from lower-order coefficients. We see that we can approximate $(w_1(x), w_2(x))$ up to order 2 by choosing $w_1(0), w_2(0)$ for order 0, then $p_6(0), p_5(0)$ for order 1 and $w_3(0), w_4(0)$ for order 2.

We can also look at $N = 8$:
\begin{equation}
    \begin{split}
        w_1(x) & = w_1(0) + p_8(0) x + (\dots + w_3(0))\ x^2 + (\dots - \frac13 p_6(0))\ x^3   +\mathcal{O}(x^4) \\
        w_2(x) & = w_2(0) + (-p_8(0) w_1(0) - p_7(0)) x + (\dots + w_4(0))\ x^2 + (\dots - \frac13 p_5(0))\ x^3  +\mathcal{O}(x^4)
    \end{split}
\end{equation}
Thus, we can approximate $(w_1, w_2)$ up to order $3$ by choosing the initial conditions $w_1(0),w_2(0), p_8(0), p_7(0),\\ w_3(0),w_4(0), p_6(0), p_5(0)$.

\section{Grading of variables}\label{sec:GradingVariables}

The examples from the preceding section show that there seems to be a 
triangular structure in the derivatives of $w_i$ and that at certain 
orders only some variables can appear. We will make this idea precise by 
defining a \textit{grading} of the variables (which will depend on the 
considered BKM system). We show that the Hamiltonians defining the 
Stäckel system are (weighted-)homogeneous and describe how the degree 
changes by taking derivatives. This will facilitate proving the 
solvability of the jets with respect to the initial conditions of the 
Stäckel system. 
    
\subsection{Defining gradings}
    
We now define the grading for the classes of BKM systems we will consider, namely (i) $\deg(m)=0$ and $n$ components and (ii) $m(\mu) = \mu,\, n = 1$ component. \textbf{NB:} The notions of degree and homogeneous are understood to be with respect to this grading, unless otherwise specified.
    
\begin{definition}[\textbf{Degree of a rational function in $w,p, c$}]\label{grading} \quad \\
We define the \textit{degree} of the single variables $w_i, p_i, c_j$ via
    \begin{align*}
    	&\deg(w_i) := i, & &\deg(p_i) := N+1-i+ \alpha, & & \deg(c_j) = j .&
    \end{align*}
Here, the constant $\alpha$ depends on the BKM system. For $m\equiv 1$ we choose $
\alpha = \frac{n}{2}$, for linear $m$ and $n=1$ we will choose $\alpha=0$. We extend 
the degree to sums and products in the usual way. If $f,g$ are homogeneous 
polynomials, we define
\begin{align*}
    \deg\left(\frac{f}{g}\right) := \deg(f) - \deg(g).
\end{align*}
We use the convention that the function $0$ has \textit{any} degree, so that the set of all polynomials of a given degree forms a vector space. 
\end{definition}
    
\begin{example}[\textbf{Degree of kinetic energy}] \label{ex:DegreeKineticEnergy}
    	We determine the degree of the kinetic energy $K_1$ of the Stäckel system, which does not depend on the polynomial $c(\mu)$.\footnote{Recall that $w_0=1$, which is a polynomial of degree zero.} We have that
\begin{align*}
K_1 &= \frac{1}{2}  \sum_{i=0}^{N-1} w_i \cdot \sum_{\substack{l,k \\l+k=N+1+i}} p_l p_k.
\end{align*}
This is a polynomial in $w,p$ and we can see that it is homogeneous of degree $N+1+2\alpha$. Indeed, for each $i \in \{0,1\ldots, N-1\}$ the degree of a term in the $i$th summand is:
\begin{align}
    	\deg(w_i) + \deg(p_l)+ \deg(p_k) = i +(N+1 + \alpha-l) + (N+1+\alpha-k) = N+1+2\alpha. \label{eq:DegreeKineticEnergy}
\end{align}
\end{example}
    
The following lemma explains how the degree changes if we differentiate the functions $w_i, p_i$ by using a homogeneous Hamiltonian $H$.
    
\begin{lemma}\label{lem:DegreeDerivatives}
    	Let $H=H(w_1,\ldots, w_N, p_1, \ldots, p_N)$ be a homogeneous function of degree $k$ and denote
    	\begin{align*}
    		& (w_i)_x = \frac{\partial H}{\partial p_i},& &(p_i)_x = -\frac{\partial H}{\partial w_i}.&
    	\end{align*}
    	Then we have
    	\begin{align*}
    		&\deg((w_i)_x) = k-(N+1-i+\alpha),& &\deg((p_i)_x) = k-i.&
    	\end{align*}
\end{lemma}
    
We aim to show that the Hamiltonian $H_1$ is homogeneous. In the two cases we 
consider, the potential is homogeneous and the choice of $\alpha$ will ensure that 
the Hamiltonian is homogeneous. The following statement will be proven in the next subsection:
    
\begin{prop}[\textbf{Degree of potential and Hamiltonian}] \label{prop:ChangeDegreeDerivatives}\quad
Consider a BKM system with either (i) $\deg(m)=0$ and $n$ components or (ii) $m(\mu) = \mu$ and $n = 1$ component. Let $T^{*}\mathbb{R}^N$ be the integrable Stäckel system constructed from the data $(N, c(\mu), m(\mu), n)$, where $c(\mu)$ is of the form $c(\mu) = \mu^{2N+n} + m(\mu) \cdot \sum_{i = 1}^N d_i \mu^{N-i}$. 
    \begin{enumerate}
    		\item \textbf{Case (i):} With respect to the grading \eqref{grading} with $\alpha = \frac{n}{2}$, we have $\deg(V_1)= N+n+1 = \deg(H_1)$.
    		\item \textbf{Case (ii):} With respect to the grading \eqref{grading} with $\alpha = 0$, we have $\deg(V_1)= N+1 = \deg(H_1)$.
    \end{enumerate}
\end{prop}

From the above proposition, we immediately obtain the following proposition, which describes how the degree of homogeneous functions behaves with respect to differentiation.

\begin{prop}[\textbf{Differentiation and degree}] \label{prop:DifferentiationDegree} \quad
Under the same assumptions as in \cref{prop:ChangeDegreeDerivatives}, the following holds.
\begin{enumerate}
	\item \textbf{Case (i):} With respect to the grading \eqref{grading} with $\alpha = \frac{n}{2}$, we have for any homogeneous function $h$ and $k \in \mathbb{N}$ that
	\begin{equation}
    	\text{deg}(h_{kx}) = \frac{n}{2} \cdot k + \deg(h).
	\end{equation}
	\item \textbf{Case (ii):} With respect to the grading \eqref{grading} with $\alpha = 
    0$, we have for any homogeneous function $h$ and $k \in \mathbb{N}$ that
	\begin{equation}
		\text{deg}(h_{kx}) = \deg(h).
	\end{equation}
\end{enumerate}
\end{prop}
\begin{proof}
	It is enough to prove the statement for the single variables and one derivative. By application of Lemma~\ref{lem:DegreeDerivatives} and the definition of the grading, the result readily follows. More precisely:
	\begin{enumerate}
		\item In this case we have
		\begin{align*}
			\deg((w_i)_x) &= (N+1+n)-(N+1-i+\frac{n}{2}) = i + \frac{n}{2} = \deg(w_i) + \frac{n}{2},
			\\
			\deg((p_i)_x) &= N+1+n-i = (N+1+\frac{n}{2} -i) + \frac{n}{2} = \deg(p_i)+ \frac{n}{2}.
		\end{align*}
		\item  In this case we have
		\begin{align*}
			\deg((w_i)_x) &= (N+1)-(N+1-i) = i = \deg(w_i),
			\\
			\deg((p_i)_x) &= N+1-i = \deg(p_i).
		\end{align*}
	\end{enumerate}
\end{proof}

\subsection{Properties of potential}

From now on, we restrict to the following special case of polynomial $c(\mu)$:
\begin{align*}
c(\mu) := \mu^{2N+n} + m(\mu) \cdot d(\mu) + r(\mu),
\end{align*}
where $d(\mu) = \sum_{i=1}^{N} d_i \mu^{N-i}$ is a polynomial of degree at most $N$ and $r(\mu)$ is a polynomial of degree strictly less than $\mathrm{deg}(m)$.
\\
In order to compute the potentials, we need the function
\begin{align*}
f(\mu) &= \frac{c(\mu)}{m(\mu)} = \underbrace{\frac{\mu^{2N+n}}{m(\mu)}}_{=:\tilde{f}(\mu)} + d(\mu) + \underbrace{\frac{r(\mu)}{m(\mu)}}_{\tilde{r}(\mu)}.
\end{align*}
For the potential we have to calculate $f(M)$ and write it as a linear combination of \linebreak $M^{N-1}, \ldots, M, \mathrm{Id}$. We have 
\begin{align*}
f(M) &= \tilde{f}(M)+ \sum_{i=1}^{N} d_i M^{N-i} + \tilde{r}(M),
\end{align*}
and so the constants $d_1, \ldots, d_N$ allow us to add any constant value to each Hamiltonian. 

As these additive constants do not change the equations of 
motions, the dynamics is only influenced by the part of the 
potential coming from $\tilde{f}(M) + \tilde{r}(M)$. 
However, these constants allow us to choose arbitrary 
initial conditions for $w_1, \ldots, w_N, p_1, \ldots, p_N$ 
as in \cref{cor:fin_gap_entire_cotangent_bundle} (and so we do not need to consider the '
common zero-energy level). In the cases that 
we investigate, $\tilde{f}(M)$ is just a power of $M$. More 
precisely, for case (i) we have that $\tilde{f}(M) = 
M^{2N+n}$ and for case (ii) we have that $\tilde{f}(M) = M^{2N+n-1}$. In order to 
investigate the potential energy $V_1$ for both cases simultaneously, we study $M^k$ for $k \in \N$. Note that due to Cayley--Hamilton theorem there is a simple recursion between the coefficients of $M^k$ with respect to the basis $M^{N-1}, \ldots, M, \mathrm{Id}$. We introduce a notation for these coefficients, as follows.

\begin{definition}[\textbf{Constants $v_{i,k}$}] \quad \\ 
	We define the constants $v_{i,k}$ for $i \in \{1,\ldots, N\}, k \in \N$ via
	\begin{align*}
		M^k =: \sum_{i=1}^{N} v_{i,k} M^{N-i}.
	\end{align*}
\end{definition}

In the case (i) with $m=1$ we have $r=0$ and hence the potential $V_1$ is (up to the constant $d_1$) just $v_{1,2N+n}$ and in the case $m(\mu)=\mu$ we have that $V_1$ is $v_{1,2N+n-1}$ together with a contribution of $\tilde{r}(M)$. So we investigate $v_{i,k}$ in order to prove properties of $V_1$. We begin with recursion formulas for the $v_{i,k}$:

\begin{lemma}[\textbf{Initial values and recursion formula for $v_{i,k}$}] \quad \label{lem:Recursion_vik}
	\begin{enumerate}
		\item For $k < N$ we have: $v_{i,k} = \delta_{i, N-k}$.
		\item For $k=N$ we have $v_{i, N} = -w_i$.
		\item For $k\geq N$ we have the recursion formula:
		\begin{align*}
		v_{i, k+1} &= \begin{cases}
		v_{i+1, k} - v_{1,k} w_i \quad & i<N \\
		-v_{1,k} w_N \quad & i=N.
		\end{cases}
		\end{align*} 
	\end{enumerate}
\end{lemma}

\begin{proof}
	(1) is just the definition of $v_{i,k}$, whereas (2) is a direct consequence of the Cayley--Hamilton theorem. For (3) we can calculate using the Cayley--Hamilton theorem:
	\begin{align*}
	M^{k+1} &= M \cdot M^{k} = M \cdot \sum_{i=1}^{N} v_{i,k} M^{N-i} = \sum_{i=2}^{N} v_{i,k} M^{N+1-i} + v_{1,k} M^N = \\ & = \sum_{i=1}^{N-1} v_{i+1,k} M^{N-i} - v_{1,k}\sum_{i=1}^{N} w_i M^{N-i}.
	\end{align*}
\end{proof}
	 
With this we can prove the following lemma:

\begin{lemma}[\textbf{Degree of $v_{i,k}$}] \quad \\
	For each $i \in \{1,\ldots, N\}, k \in \N$, $v_{i,k}$ is a homogeneous polynomial in $w_1, \ldots, w_N$ of degree $i+k-N$.
\end{lemma}
	 
The proof is a simple induction on $k$ and hence omitted. But this statement implies that $V_1$ is a homogeneous polynomial, concluding the proof of Proposition~\ref{prop:ChangeDegreeDerivatives}.
\\	 
As we have proven that $V_1$ is a homogeneous polynomial, next we ask which monomials appear in $V_1$ with non-vanishing coefficients. Note that the degree of $V_1$ is greater than $N$ in the considered cases, so all monomials need to have at least two factors $w_a\cdot w_b$, for which $a+b$ equals $\deg(V_1)$. We will focus on these monomials and prove that they all appear in the potential $V_1$. To this end, we prove properties for all $v_{i,k}$ in the following two lemmas. We start with the case of $k \in \{N, \ldots, 2N\}$. Afterwards, we do the case of $k \in \{2N, \dots, 3N\}$. We introduce the notation to capture error terms for which all terms are at least a given degree (in the usual sense of polynomials in multiple variables). For later use, we also include the $p$ variables into this notation.

\textbf{Notation}. For a polynomial $P$ in $w_1, \dots, w_N, p_1, \dots, p_N$, we use the notation $P=Q + \mathcal{O}(d)$ to indicate that the remaining terms all have at least degree $d$ in the usual sense, not with respect to the grading; i.e. we count the number of (not necessarily distinct) factors in a monomial. That is, $Q$ contains all monomial terms of $P$ with at most $d-1$ factors. If we want to indicate that there are only factors from a specific set of variables, we indicate this with a subindex, e.g. $\mathcal{O}_w(d)$.

\begin{example}
	If $P = w_1 \cdot w_{N-1} + w_N + p_{2} p_N$, then we have
	$P = w_N + \mathcal{O}(2) = w_N + p_{2} p_N + \mathcal{O}_w(2)$.
\end{example}

\begin{lemma}[\textbf{Calculations for $k \in \{N, \ldots, 2N\}$}] \quad \\
	For $k \in \{N, \ldots, 2N\}$, we have that
	\begin{align*}
	v_{i,k} = \sum_{\substack{a,b \in \{0,\ldots, N\} \\ a+b=i+k-N, \\ b \leq k-N}} w_a \cdot w_b + \mathcal{O}_w(3)
	\end{align*}
	where we use the convention $w_0 := -1$ (\textbf{NB:} This is different from our usual convention for $w_0$). 
\end{lemma}
\begin{proof}
The proof is by induction on $k$. The base case $k=N$ is treated in the derivation of the recursion formula: $v_{i,N} = -w_i$. 

To prove the inductive step, assume the statement is true for $k$. Then for $c_{i, k+1}$ we can use the recursion formula and distinguish the cases $i < N$ and $i=N$:
\begin{itemize}
	\item For $i<N$ we have
		\begin{align*}
		v_{i, k+1} = v_{i+1, k} - v_{1,k} \cdot w_i.
		\end{align*}
	By the induction hypothesis, $v_{i+1, k} = \sum\limits_{\substack{a,b \in \{0,\ldots, N\} \\ a+b=i+1+k-N, \\ b \leq k-N}} w_a \cdot w_b + \mathcal{O}_w(3)$. Again using the induction hypothesis, $v_{1,k} = -w_{k+1-N} + \mathcal{O}_w(2)$ (the term $-w_{k+1-N}$ appears for $k+1-N, b=0$).
	As $\mathcal{O}_w(2) \cdot w_i = \mathcal{O}_w(3)$, these terms are not important. So
	\begin{align*}
		v_{i, k+1} &= v_{i+1, k} - v_{1,k} \cdot w_i = \sum\limits_{\substack{a,b \in \{0,\ldots, N\} \\ a+b=i+1+k-N, \\ b \leq k-N}} w_a \cdot w_b + v_{1,k} \cdot w_{k+1-N} + \mathcal{O}_w(3) = \\ & = \sum\limits_{\substack{a,b \in \{0,\ldots, N\} \\ a+b=i+k+1-N, \\ b \leq k+1-N}} w_a \cdot w_b + \mathcal{O}_w(3).
	\end{align*}
	\item For $i=N$, as $a+b=i+k+1-N=k+1$ and $b \leq k+1-N$ shall hold, there is only one such pair: $a=N$ and $b=k+1-N$. This we get from 
	\begin{align*}
		v_{N, k+1} = - v_{1,k} \cdot w_N = -(-w_{k+1-N} + \mathcal{O}_w(2)) \cdot w_N = w_{k+1-N} \cdot w_N + \mathcal{O}_w(3).
	\end{align*}
	\end{itemize}
\end{proof}

\begin{remark}
	In the next case $k \in \{2N, \ldots, 3N\}$, as we know that the $v_{i,k}$ are polynomials of degree $i+k-N \in \{N+1, \ldots, 3N\}$, we know that the terms with the smallest number of $w$-factors have to be quadratic or cubic. We are mainly interested in the quadratic part and in view of the degree of $v_{i,k}$ we restrict to $i \leq 3 N - k$.
\end{remark}

\begin{lemma}[\textbf{Calculations for $k \in \{2N, \ldots, 3N\}$}] \label{lem:Calculationsk2N3N} \quad \\
	For each $k \in \{2N, \ldots, 3N\}$ and each $i \leq 3N-k$, we have that
	\begin{align*}
	v_{i,k} = \sum_{\substack{a,b \\ a+b=i+k-N}} w_a \cdot w_b + \mathcal{O}_w(3).
	\end{align*}
\end{lemma}

\begin{proof}
	\quad \\
	Again, the proof is by induction on $k$. The base case is covered by the previous lemma. So we only need to do the induction step from $k$ to $k+1$. Thus, we consider an index $i$ with $i \leq 3N-(k+1)$. 
    
	From the recursion formula (note that $i<N$ has to hold) we obtain
	\begin{align*}
	v_{i, k+1} = v_{i+1, k} - v_{1,k} w_i = \sum_{\substack{a,b \\ a+b=i+1+k-N}} w_a \cdot w_b + \mathcal{O}_w(3) - \mathcal{O}_w(2) w_i = \sum_{\substack{a,b \\ a+b=i+(k+1)-N}} w_a \cdot w_b + \mathcal{O}_w(3).
	\end{align*}
\end{proof}

\section{Proof for $\mathrm{deg}(m) = 0$}
\label{sec:Proofm=1}

In the KdV and KB case of Subsection~\ref{subsec:Examplesm=1} we saw that Taylor coefficients of suitable orders $k$ satisfy the following triangular structure:
\begin{equation}
    (w_i)_{kx} = a_k \cdot z_{new} + F(z_{old}),   \hspace{2cm} (a_k \in \mathbb{Q} \backslash \{0\})
\end{equation}
where $z_{new}$ is a new variable that enters linearly and $F(z_{old})$ is 
a polynomial in a number of variables $z_{old}$ that already have been 
fixed at either (i) lower-order coefficients or (ii) $k$th order coefficients 
of a position component with index strictly smaller than $i$. Since we can use the polynomial $c(\mu)$ to set the Hamiltonians to any value, we have a freedom of $2N$ variables (choosing the initial point on $T^{*}\R^N_{(w,p)}$). For the above triangular structure, this entails that we can approximate up to order $k = \left\lfloor \frac{2N}{n} \right\rfloor-1$. We now prove this triangular structure for the general case with $\deg(m) = 0$. As a computational tool we shall rely on the grading for polynomials in variables $w_i, p_i$ cf. Section~\ref{sec:GradingVariables}.

\setcounter{theorem}{0}
\begin{theorem}[\textbf{Approximating jets of initial data for $\deg(m) = 0$}]
Consider an $n$-component BKM system with $\deg(m) = 0$. Let $\mathcal{R}: \mathbb{R}^N 
\rightarrow \mathbb{R}^n$ be the finite-reduction map for the Stäckel system with $N$ gaps. A $k$-jet of any initial data $x \mapsto u(x,0)$ can be realized as the $k$-jet of a finite-gap solution. More precisely, for any $k$-jet $j_0^ku$ there exists $N \in \mathbb{N}$ and an initial condition $(w(0), p(0))$ determining a solution $w$ such that $\mathcal{R}(j_0^kw) = j^k_0u$.
\end{theorem}

In view of the triangular structure for the reduction map $\mathcal{R}$ proved in  
Proposition~\ref{prop:ReductionMap}, we need to consider the derivatives $(w_i)_{kx}$ 
for $i=1, \ldots, n$ and $k=0, \ldots, \left\lfloor \frac{2N}{n} \right\rfloor-1$. We 
now show that there is also a triangular structure for the Taylor coefficients 
$(w_i)_{kx}$ with respect to $w,p$ coordinates (ordered by increasing degrees).

To this end, we aim to find the highest-degree variables appearing in $(w_i)_{kx}$. The homogeneity property for $(w_i)_{kx}$ implies that we can rule out variables with a degree greater than $\frac{n}{2} \cdot k + i$. However, we consider the question whether the variables with this particular degree appear. If they do, they necessarily appear linearly. Our goal is to prove the following proposition, which directly implies the aforementioned theorem:

\begin{prop}[\textbf{Linear terms in Taylor coefficients}] \label{prop:LinearTerms}
	Let $i \in \{1, \ldots, n\}$ and $k \in \left\{0, \ldots, \left\lfloor\frac{2N}{n}-1\right\rfloor \right\}$.
	\begin{itemize}
		\item For even $k$, we have that
        \begin{equation}
            (w_i)_{kx} = a_k \cdot w_{n \cdot \frac{k}{2} +i} + \mathcal{O}(2) \hspace{2.08cm} (a_k \neq 0).
        \end{equation}
		\item For odd $k$, we have that
        \begin{equation}
            (w_i)_{kx} = a_k \cdot p_{N+1-(n \cdot \frac{k-1}{2} +i)} + \mathcal{O}(2) \qquad (a_k \neq 0).
        \end{equation}
	\end{itemize}
\end{prop}

The idea of the proof is that a term of highest possible degree has to appear linearly. 
Such a term must arise from differentiating the leading linear term in $(w_i)_{(k-1)x}$.
We do an induction on $k$ and see that each highest-degree linear term in $(w_i)_{(k-1)x}$ produces only one linear term in $(w_i)_{kx}$.

\begin{proof}[Proof of Proposition~\ref{prop:LinearTerms}] \quad \\
    The proof is by induction on $k$. The base case $k=0$ is trivial, as $(w_i)_{0x}= w_i$. \\
	So let us prove the inductive step from $k$ to $k+1$. As $(w_i)_{kx}$ contains only one linear term, the derivative of this term is the only possible linear term in $(w_i)_{(k+1)x}$. Note that the derivative of $\mathcal{O}(l)$ is again a term $\mathcal{O}(l)$. Now we split cases according to $k$ being even or odd:
	\begin{itemize}
		\item If $k$ is even, $(w_i)_{kx} = a_k \cdot w_{n \cdot \frac{k}{2} +i} + \mathcal{O}(2)$. So
		\begin{align*}
			(w_i)_{(k+1)x} &= a_k \cdot \left(w_{n \cdot \frac{k}{2} +i}\right)_{x} + \mathcal{O}(2) = a_k \frac{\partial K_1}{\partial p_{n \cdot \frac{k}{2} +i}} + \mathcal{O}(2).
		\end{align*}
		Recall that $K_1 = -\frac{1}{2}  \sum_{i=0}^{N-1} w_i \cdot \sum_{\substack{l,k \\l+k=N+1+i}} p_l p_k$ with $w_0=1$. So for any index $l \in \{1, \ldots, N\}$,
		\begin{align*}
			\frac{\partial K_1}{\partial p_l} = -\sum_{\substack{i,k\\ l+k=N+1+i}} w_i p_k = -p_{N+1-l} + \mathcal{O}(2).
		\end{align*}
		Replacing $l= n \cdot \frac{k}{2} +i$, we obtain
		\begin{align*}
			(w_i)_{(k+1)x} &= -a_k p_{N+1-( n \cdot \frac{(k+1)-1}{2} +i)} + \mathcal{O}(2).
		\end{align*}
		\item If $k$ is odd, $(w_i)_{kx} = a_k \cdot p_{N+1-(n \cdot \frac{k-1}{2} +i)} + \mathcal{O}(2)$. So
		\begin{align*}
			(w_i)_{(k+1)x} &= a_k (p_{N+1-( n \cdot \frac{k-1}{2} +i)})_{x} + \mathcal{O}(2) = a_k \left(\frac{\partial K_1}{\partial w_{N+1-( n \cdot \frac{k-1}{2} +i)}} + \frac{\partial V_1}{\partial w_{N+1-( n \cdot \frac{k-1}{2} +i)}}\right) + \mathcal{O}(2).
		\end{align*}
		As each $w$-variable appears in $K_1$ only in $\mathcal{O}(3)$-terms, all terms in $\frac{\partial K_1}{\partial w_{N+1-( n \cdot \frac{k-1}{2} +i)}}$ are $\mathcal{O}(2)$. So we only need to consider the second term and there we can ignore the $\mathcal{O}(3)$ terms. Here we can use Lemma~\ref{lem:Calculationsk2N3N}, as $V_1 = v_{1,2N+n}$, and so $V_1 = \frac{1}{m_0}\sum_{\substack{a,b \\ a+b=N+n+1}} w_a \cdot w_b + \mathcal{O}(3)$. We obtain
		\begin{align*}
			\frac{\partial V_1}{\partial w_{N+1-( n \cdot \frac{k-1}{2} +i)}} &= \frac{2}{m_0} w_{N+n+1-(N+1-( n \cdot \frac{k-1}{2} +i))} +\mathcal{O}(2) = \frac{2}{m_0}w_{n \cdot \frac{k+1}{2} +i} +\mathcal{O}(2).
		\end{align*}
		Combining this with the above calculation, the claim is proven.
	\end{itemize}
\end{proof}

So if we consider the vector-valued function $(w_1, w_2, \ldots, w_n, (w_1)_x, \ldots, (w_n)_x, (w_1)_{2x}, \ldots (w_n)_{kx})^{T}$ and differentiate it with respect to the variables $(w_1, \ldots, w_n, p_N, \ldots, p_{N-n+1}, w_{n+1}, \ldots w_{2n}, p_{N-n}, \ldots)$, we get a triangular matrix. Additionally, we can solve the system of equations in these variables choosing them in the right order and prescribe the values of $w_1, \ldots, (w_n)_{kx}$ as desired. This concludes the proof of Theorem~\ref{thm:ApproximationCasem=1}.

\section{Proof for $\mathrm{deg}(m) = 1, n = 1$}
\label{sec:ProofDegm=1}

Now we consider the case of one-component systems with an arbitrary (non-constant) polynomial $m(\mu) = \mu + m_0$. In this case, we can prove the following theorem:

\begin{theorem}[\textbf{Approximating jets of initial data for $\deg(m) = 1, n=1$ on an open set}] 
	Consider a $1$-component BKM system with $\deg(m) = 1$. Let $\mathcal{R}: \mathbb{R}^N \rightarrow \mathbb{R}$
	be the finite-reduction map for the Stäckel system with $N$ gaps. For an open set of $k$-jets $j_0^ku$ there exists $N \in \mathbb{N}$ and an initial condition $(w(0), p(0))$ determining a solution $w$ such that $\mathcal{R}(j_0^kw) = j^k_0u$.
    
    Moreover, the BKM system and finite-reduction map may be viewed over the complex numbers. Over $\mathbb{C}$, the statement holds true for a Zariski-open (in particular: dense) set of $k$-jets.
\end{theorem}

We split the proof into two parts. First, we will consider the case $m(\mu) = 
\mu$ of the Camassa-Holm equation. In this case the potential is homogeneous with respect to our grading, allowing an approach similar to the case $\deg(m)=0$. This will be done in Section~\ref{subsec:ProofCamassaHolmCase}. 
Afterwards, we reduce the general case $m(\mu) = \mu + m_0$ to this special 
case by a coordinate change and a suitable choice of the polynomial $c(\mu)$ 
in Section~\ref{subsec:GeneralizationArbitraryPolynomial}.

\subsection{Proof for Camassa-Holm case $m(\mu) = \mu$} \label{subsec:ProofCamassaHolmCase}

We look at the Camassa--Holm equation, which is 
of BKM type III with $n = 1$ and polynomial $m(\mu) = \mu$. The Camassa--Holm equation is given 
by
\begin{equation}
    \begin{split}
        u_t & = 2 u q_x + q u_x   \\
        u & = 2 q - q_{xx} 
    \end{split}
\end{equation}
In order to prove Theorem~\ref{thm:ApproximationCasemlinear} in the case $m(\mu) = \mu$, we 
adapt the approach from the previous section. To this end, note first that the 
reduction map $\mathcal{R}$ is given by $u_1 = \mathcal{R}(w) = \frac{c_{2N+1}}{w_N^2}$. So 
it is sufficient to approximate $w_N$ up to any order (the constant $c_{2N+1}$ ensures we can approximate initial data of either sign). In view of the reduction map, we cannot choose $c_{2N+1}=0$ as we did 
in the previous section. Our approach for the polynomial $c$ is to take
\begin{align*}
	c(\mu) = \mu^{2N+1} + m(\mu) \sum_{i=1}^{N} d_i \mu^{N-i} + c_{2N+1}.
\end{align*}
Again, the constants $d_1, \ldots, d_N$ are used to get arbitrary values for the potentials 
so that we are no longer restricted to the zero-surface of the Hamiltonians cf. 
\cref{cor:fin_gap_entire_cotangent_bundle}. The constants $d_i$ do not enter the equations of motion, and so we consider the following expressions:
\begin{align*}
	H_1  & = K_1 + V_1, \text{ where} \\
	K_1 &=  -\frac{1}{2}  \sum_{i=0}^{N-1} w_i \cdot \sum_{\substack{l,k \\l+k=N+1+i}} p_l p_k \quad  \text{with } w_0=1, \\
	V_1 &= \frac{c_{2N+1}}{w_N} + \sum_{i=1}^{N} w_i w_{N+1-i} + R(w_1, \ldots, w_{N-1}).
\end{align*}
Note that the fractional part $\frac{c_{2N+1}}{w_N}$ comes from $c_{2N+1} M^{-1}$, the contribution of $M^{2N}$ gives the other summands, where we used Lemma~\ref{lem:Calculationsk2N3N}:  $R$ is $\mathcal{O}_w(3)$ (i.e., each term contains at least three factors from $w$) and is polynomial of homogeneous degree $N+1$ in $w_1, \ldots, w_{N-1}$. In particular, $R$ does not depend on $w_N$. We can consider this as a homogeneous potential of degree $N+1$, keeping in mind that $\deg(c_{2N+1}) = 2N+1$.

With this setup, we consider for fixed $N$ and $k=2N-2$ (the approximation order) the vector
\begin{equation}
    (w_N, (w_N)_{x}, \ldots, (w_N)_{k\cdot x})^T \in \R^{2N-1}.
\end{equation}
In contrast to the case $m \equiv 1$, each component $(w_N)_{jx}$ has 
the same degree $N$ (in view of 
Proposition~\ref{prop:DifferentiationDegree}), and so we cannot 
guarantee that new variables appear at each differentiation step.

Therefore we instead aim to compute the Jacobian of this function of $(w,p)$ with respect to the variables $(w_1, \ldots, w_{N-1}, p_1, \ldots, p_{N-1})$ (we choose a more convenient ordering later). This is a rational matrix in $(w,p)$. We would like this matrix to be invertible at a suitable point, so that we get $k$-jet surjectivity onto an open set in $J^k$. We shall show that the Jacobi matrix evaluated at $0 = w_1 = w_2 = \ldots = w_{N-1} = p_1 = \ldots = p_N$ is nondegenerate for each $N$. 

\textbf{Notation}. Let $\varphi$ be a function or a matrix. We denote by $\varphi|_{\mathrm{all}\ 0}$ the result obtained from substituting $0 = w_1 = \dots = w_{N-1} = p_1 = \dots = p_N$ (leaving only $w_N$ remaining as well as the constant $c_{2N+1}$). 

By the rationality of the Taylor coefficients, nondegeneracy of Jacobian at one point implies that the matrix is invertible for almost every point. Over the reals, however, this does not imply full $k$-jet surjectivity.

We now reorder the variables $(w_1, \dots, w_{N-1}, p_1, \dots, p_{N-1})$ according to increasing degree.
\begin{definition}[\textbf{Variable ordering}] \quad \\
	We define the variables $v_1, \ldots, v_{2N-2}$ by
	\begin{align*}
		v_j := \begin{cases}
			w_{\frac{j+1}{2}} \quad & \text{$j$ odd} \\
			p_{N-\frac{j}{2}} \quad & \text{$j$ even}
		\end{cases}.
	\end{align*}
	We are interested in the Jacobian matrix
    \begin{equation}
        J:=\left(\frac{\partial (w_N)_{i\cdot x}}{\partial v_j}_{\mid \text{all 0}} \right)_{i,j = 1\ldots, 2N-2}.
    \end{equation}
\end{definition}

\begin{remark}
	We can use $w_N(0)$ to fix the constant term. This is why we exclude $w_N$ from the variables and start with the first derivative. This ordering of the variables will lead to an anti-triangular structure, see Example \ref{ex:AntitriangularStructure}. One could get a usual lower triangle structure by reversing the ordering. We stick to the choice of ordering variables with increasing degree.
\end{remark}

We now start with the computations. In order to illustrate what we prove in upcoming lemmas we shall make use of explicit examples in the case $N = 6$ and indicate with various symbols which entries we can compute (for general $N$). 

The Jacobian $J$ for $N = 6$ has the following form:
{ \small
	\begin{equation}\label{ex:AntitriangularStructure}
		J = \begin{pmatrix}
			\circled{0} & \circled{0} & \circled{0} & \circled{0} & \circled{0} & \circled{0} & \circled{0} & \circled{0} & \dashedcircle{0} & \dashedcircle{-1} 
			\\
			\circled{0} & \circled{0} & \circled{0} & \circled{0} & \circled{0} & \circled{0} & \circled{0} & \circled{0} & \dashedbox{\frac{c_{13}}{w_6^{2}}} & 0 
			\\
			\circled{0} & \circled{0} & \circled{0} & \circled{0} & \circled{0} & \circled{0} & \dashedcircle{0} & \dashedcircle{-\frac{2 c_{13}}{w_6^{2}}} & 0 & -2 
			\\
			\circled{0} & \circled{0} & \circled{0} & \circled{0} & \circled{0} & \circled{0} & \dashedbox{\frac{2 c_{13}^{2}}{w_6^{4}}} & 0 & \frac{2 c_{13}}{w_6^{2}} & 0 
			\\
			\circled{0} & \circled{0} & \circled{0} & \circled{0} & \dashedcircle{0} & \dashedcircle{-\frac{4 c_{13}^{2}}{w_6^{4}}} & 0 & \frac{8 c_{13}}{w_6^{2}} & 0 & -4 
			\\
			\circled{0} & \circled{0} & \circled{0} & \circled{0} & \dashedbox{\frac{4 c_{13}^{3}}{w_6^{6}}} & 0 & -\frac{28 c_{13}^{2}}{w_6^{4}} & 0 & \frac{4 c_{13}}{w_6^{2}} & 0 
			\\
			\circled{0} & \circled{0} & \dashedcircle{0} & \dashedcircle{-\frac{8 c_{13}^{3}}{w_6^{6}}} & 0 & \frac{136 c_{13}^{2}}{w_6^{4}} & 0 & -\frac{248 c_{13}}{w_6^{2}} & 0 & -8 
			\\
			\circled{0} & \circled{0} & \dashedbox{\frac{8 c_{13}^{4}}{w_6^{8}}} & 0 & -\frac{248 c_{13}^{3}}{w_6^{6}} & 0 & \frac{1032 c_{13}^{2}}{w_6^{4}} & 0 & \frac{8 c_{13}}{w_6^{2}} & 0 
			\\				\dashedcircle{0} & \dashedcircle{-\frac{16 c_{13}^{4}}{w_6^{8}}} & 0 & \frac{832 c_{13}^{3}}{w_6^{6}} & 0 & -\frac{7776 c_{13}^{2}}{w_6^{4}} & 0 & \frac{11072 c_{13}}{w_6^{2}} & 0 & -16 
			\\
			\dashedbox{\frac{16 c_{13}^{5}}{w_6^{10}}} & 0 & -\frac{1264 c_{13}^{4}}{w_6^{8}} & 0 & \frac{20496 c_{13}^{3}}{w_6^{6}} & 0 & -\frac{61424 c_{13}^{2}}{w_6^{4}} & 0 & \frac{16 c_{13}}{w_6^{2}} & 0 
			\end{pmatrix}
		\end{equation}
	}

We now prove the anti-triangular structure of \cref{ex:AntitriangularStructure} in general. We formulate it as the following proposition and the proof of this proposition will be given in the rest of the section as a combination of lemmas 6-10. Note also that the theorem for Camassa--Holm follows directly from this proposition.

\begin{prop}[\textbf{Structure of the Jacobian}]\label{prop:structure_jacobian} \quad \\
	For each $i,j \in \{1, \ldots, 2N-2\}$ we have:
	\begin{itemize}
		\item If $j<2N-1-i$, then $\left(\frac{\partial (w_N)_{i\cdot x}}{\partial v_j}\right)_{\mid \text{all 0}} = 0$.
		\item If $j=2N-1-i$, then $\left(\frac{\partial (w_N)_{i\cdot x}}{\partial v_j}\right)_{\mid \text{all 0}} \neq 0$, more precisely it is $(-1)^{i} \cdot 2^{\lfloor \frac{i-1}{2} \rfloor} \cdot \left( \frac{c_{2N+1}}{w_N^2} \right)^{\lfloor \frac{i}{2} \rfloor}$.
	\end{itemize}
\end{prop}

In order to prove the proposition, it is useful to consider the following terms $((w_i)_{jx})_{\mid \text{all 0}}$ and $((p_{N+1-i})_{jx})_{\mid \text{all 0}}$ for $i=1,\ldots, N-1$ and $j=1,\ldots, 2N-2$. This is because the derivatives of $w_N$ will consist of products of these terms. Using the product rule for differentiation with respect to $v_j$ we can ignore some of these terms if they are zero (after the evaluation).

\begin{example}\label{example:vars_at_all0}
	We consider these expressions in the example $N=6$. We depict the matrix with entry $(i,j)$ being $((w_i)_{jx})_{\mid \text{all 0}}$ respectively for the momenta. It is important that these matrices are not Jacobians, as we do not differentiate with respect to the variables.
	\begin{align*}
		((w_i)_{jx})_{\mid \text{all 0}} &= \begin{pmatrix}
			\circled{0}& \boxed{\frac{c_{13}}{w_6^{2}}} &0&-\frac{2 c_{13}}{w_6^{2}}&0&\frac{60 c_{13}}{w_6^{2}}&0&-\frac{2056 c_{13}}{w_6^{2}}&0&\frac{122864 c_{13}}{w_6^{2}} \\
			\circled{0} & \circled{0}
            &\circled{0}&\boxed{\frac{2 c_{13}^{2}}{w_6^{4}}}&0&-\frac{36 c_{13}^{2}}{w_6^{4}}&0&\frac{1528 c_{13}^{2}}{w_6^{4}}&0&-\frac{102416 c_{13}^{2}}{w_6^{4}} \\
			\circled{0}&\circled{0}&\circled{0}&\circled{0}&\circled{0}&\boxed{\frac{4 c_{13}^{3}}{w_6^{6}}}&0&-\frac{264 c_{13}^{3}}{w_6^{6}}&0&\frac{23024 c_{13}^{3}}{w_6^{6}} \\
			\circled{0}&\circled{0}&\circled{0}&\circled{0}&\circled{0}&\circled{0}&\circled{0}&\boxed{\frac{8 c_{13}^{4}}{w_6^{8}}}&0&-\frac{1296 c_{13}^{4}}{w_6^{8}}\\
			\circled{0}&\circled{0}&\circled{0}&\circled{0}&\circled{0}&\circled{0}&\circled{0}&\circled{0}&\circled{0}&\boxed{\frac{16 c_{13}^{5}}{w_6^{10}}}
		\end{pmatrix}
		\\
		((p_{N+1-i})_{jx})_{\mid \text{all 0}} &= \begin{pmatrix}
			\boxed{-\frac{c_{13}}{w_6^{2}}}&0&\frac{2 c_{13}}{w_6^{2}}&0&-\frac{60 c_{13}}{w_6^{2}}&0&\frac{2056 c_{13}}{w_6^{2}}&0&-\frac{122864 c_{13}}{w_6^{2}}&0 \\
			\circled{0}&\circled{0}&\boxed{\frac{c_{13}^{2}}{w_6^{4}}}&0&\frac{6 c_{13}^{2}}{w_6^{4}}&0&\frac{292 c_{13}^{2}}{w_6^{4}}&0&-\frac{15304 c_{13}^{2}}{w_6^{4}}&0 \\
			\circled{0}&\circled{0}&\circled{0}&\circled{0}&\boxed{-\frac{4 c_{13}^{3}}{w_6^{6}}}&0&-\frac{184 c_{13}^{3}}{w_6^{6}}&0&-\frac{2144 c_{13}^{3}}{w_6^{6}}&0 \\
			\circled{0}&\circled{0}&\circled{0}&\circled{0}&\circled{0}&\circled{0}&\boxed{\frac{34 c_{13}^{4}}{w_6^{8}}}&0&\frac{5376 c_{13}^{4}}{w_6^{8}}&0 \\
			\circled{0}&\circled{0}&\circled{0}&\circled{0}&\circled{0}&\circled{0}& \circled{0}&\circled{0}&\boxed{-\frac{496 c_{13}^{5}}{w_6^{10}}}&0
		\end{pmatrix}
	\end{align*}
\end{example}

We can prove the following lemma 6 which accounts for the `lower left almost-triangular' part of the matrix containing zeros (circled in the example above). In lemma 7 we compute the first nonzero entry in each row (boxed in the example above).

\begin{lemma} \label{lem:simpleDerivativesAtAll0}
	For all $i \in \{1, \ldots, N-1\}$ we have:
	\begin{itemize}
		\item For all $k<2\cdot i$: $((w_i)_{kx})_{\mid \text{all 0}} =0$.
		\item For all $k<2\cdot i-1$: $((p_{N+1-i})_{kx})_{\mid \text{all 0}}= 0$.
	\end{itemize}
\end{lemma}
\begin{proof}
	We prove the statement by strong induction on $i$. The base case is $i=1$. For $p_{N}$ there is only $k=0$ to be considered, but $(p_N)_{\text{all 0}} = 0$ holds since $p_N$ is one of the variables being set to zero (cf. the notation $|_{\mathrm{all}\ 0}$).
	For $w_1$, we have to consider $k=0$ (which is again obvious) and $k=1$, which follows from $(w_1)_x = -p_N$.

	For the induction hypothesis, assume that the statement is proven for all indices $j<i$. 
    Now we need to prove it for $i := j + 1\, (< N)$. Consider $p_{N+1-i}$ and $w_i$. The 
    statement for $k=0$ is trivial, so let us start by differentiating once. We obtain the 
    following two expressions
	\begin{equation}
    \begin{split}
		(p_{N+1-i})_{x} &= -2w_i + \frac{1}{2} \sum_{\substack{l,k \\l+k=N+1+(N+1-i)}} p_l p_k - \frac{\partial R}{\partial w_{N+1-i}}, \\
		(w_i)_{x} &= -p_{N+1-i} - \sum_{j=1}^{i-1} w_j p_{N+1-i+j}.
        \end{split}
	\end{equation}
    In order to establish vanishing of the first quantity, note that $\frac{\partial R}{\partial w_{N+1-i}}$ and $\sum_{\substack{l,k \\l+k=N+1+(N+1-i)}} p_l p_k$ are both homogeneous polynomials of degree $i$, and of type $\mathcal{O}_w(2)$ and $\mathcal{O}_p(2)$, respectively. Thus, $\frac{\partial R}{\partial w_{N+1-i}}$ can only depend on $w_1, \dots, w_{i-1}$, and applying the induction hypothesis to each factor gives vanishing at $|_{\mathrm{all}\ 0}$. Similarly for the other term. We now show vanishing of the higher derivatives of $p_{N-i+1}$ at $|_{\mathrm{all}\ 0}$ and relate these to the derivatives of 
    $w_i$. The induction hypothesis implies the following statement: For every polynomial $\Phi(w_1, \ldots, w_{i-1})$ of degree $d$ and every $l<2d$ we have
	\begin{align*}
		\left(\left(\Phi(w_1, \ldots, w_{i-1})\right)_{lx}\right)_{\mid \text{all 0}} = 0.
	\end{align*}
    This is most easily seen for monomial expressions using the product rule. Indeed, differentiating a monomial of degree $j$ less than $2j$ times results in vanishing at $|_{\mathrm{all}\ 0}$.

	A similar argument holds for higher derivatives of $\sum_{\substack{l,k \\l+k=N+1+(N+1-i)}} p_l p_k$, as follows. By the induction hypothesis, differentiating $p_{N+1-j}$ less than $2j-1$ times results in vanishing at $|_{\mathrm{all}\ 0}$. This term of degree $i$ will vanish if we differentiate less than $2i-2$ times. This condition is satisfied in view of the condition $k<2\cdot i-1$.

	So we obtain for $k$ restricted as in the lemma:
	\begin{align*}
		((p_{N+1-i})_{kx})_{\mid \text{all 0}} &= ((w_{i})_{(k-1)x})_{\mid \text{all 0}}, \\
		((w_{i})_{kx})_{\mid \text{all 0}} &= ((p_{N+1-i})_{(k-1)x})_{\mid \text{all 0}},
	\end{align*}
	as the other terms vanish by the induction hypothesis. The lemma now follows immediately from induction on $k$.
\end{proof}

In the \cref{example:vars_at_all0} we see that there are nonzero values in the matrix, indicated by a box around the entries. These correspond to $((w_i)_{(2i)x})|_{\mathrm{all}\ 0}$ and $((p_{N+1-i})_{(2i-1)x})|_{\mathrm{all}\ 0}$, respectively. In the next lemma, we show that their nonvanishing holds for general $N$.

\begin{lemma} \label{lem:SimpleDerivativesNonzeroValues} \quad
	For all $i \in \{1, \ldots, N-1\}$ we have:
	\begin{itemize}
		\item $((w_i)_{2i\cdot x})_{\mid \text{all 0}} = 2^{i-1} \cdot \left(\frac{c_{2N+1}}{w_N^2}\right)^{i}$.
		\item $((p_{N+1-i})_{(2i-1)x})_{\mid \text{all 0}}= (-1)^{i} \cdot t_i \cdot \left(\frac{c_{2N+1}}{w_N^2}\right)^{i}$, where $t_i$ are the reduced tangent numbers.
	\end{itemize}
\end{lemma}
\begin{proof}
	We start with the momenta. We shall derive a recursion formula involving the reduced tangent numbers. The base case is $\left((p_{N})_{x}\right)_{\mid \text{all 0}} = - \frac{c_{2N+1}}{w_N^2}$ as claimed. \\
	Now consider some $i>1$. Again we consider
	\begin{align*}
		(p_{N+1-i})_x &= -2w_i + \frac{1}{2} \sum_{\substack{l,k \\l+k=N+1+(N+1-i)}} p_l p_k - \frac{\partial R}{\partial w_{N+1-i}}.
	\end{align*}
	We note that from the above lemma, $\left((w_i)_{(2i-2) x} \right)_{\mid \text{all 0}} = 0$ and the same holds true for $\frac{\partial R}{\partial w_{N+1-i}}$. So the only contribution can come from $\frac{1}{2} \sum_{\substack{l,k \\l+k=N+1+(N+1-i)}} p_l p_k$, which we can rewrite as
	\begin{align*}
		\frac{1}{2} \sum_{\substack{l,k \\l+k=N+1+(N+1-i)}} p_l p_k = \frac{1}{2} \sum_{j=1}^{i-1} p_{N+1-j} p_{N+1-(i-j)}. 
	\end{align*}
    By \cref{lem:simpleDerivativesAtAll0}, $p_{N+1-j}$ vanishes at $|_{\mathrm{all}\ 0}$ if differentiated less than $2j-1$ times, and similarly $p_{N+1-(i-j)}$ if differentiated less than $2(i-j)-1$ times. Since we differentiate $(p_{N+1-i})_x$ an additional $2i - 2$ times, there is exactly one way for each summand to distribute these derivatives among the two factors without vanishing at $|_{\mathrm{all}\ 0}$. So we obtain:
	\begin{align*}
		\left((p_{N+1-i})_{(2i-1)\cdot x}\right)_{\mid \text{all 0}} &= \frac{1}{2} \sum_{j=1}^{i-1} \binom{2i-2}{2j-1} \left(\left(p_{N+1-j}\right)_{(2j-1) \cdot x}\right)_{\mid \text{all 0}} \left(\left(p_{N+1-(i-j)}\right)_{(2(i-j)-1) \cdot x}\right)_{\mid \text{all 0}}.
	\end{align*}
	Now we proceed by induction on $i$. We obtain the statement with a sequence $t_i$ that has to satisfy
	\begin{align*}
		t_i &= \frac{1}{2} \sum_{j=1}^{i-1} \binom{2i-2}{2j-1} t_j t_{i-j}.
	\end{align*}
    In \cref{lem:RecursionReducedTangentNumbers} in the appendix, we give a proof that this is one of the recursions for the reduced tangent numbers, so that the result follows.

	Next, we consider the statement for the $w_i$. Note that
	\begin{align*}
		(w_i)_{x} &= -p_{N+1-i} - \sum_{j=1}^{i-1} w_j p_{N+1-i+j} = -\sum_{j=0}^{i-1} w_j p_{N+1-i+j},
	\end{align*}
	with the convention $w_0:=1$. The base case $((w_1)_{2x})|_{\mathrm{all}\ 0} = \frac{c_{2N+1}}{w_N^2}$ follows readily from the above formula. A similar inductive 
    argument yields:
	\begin{align*}
		((w_i)_{2i\cdot x})_{\mid \text{all 0}} &= - \sum_{j=0}^{i-1} \binom{2i-1}{2j} ((w_j)_{2j\cdot x})_{\mid \text{all 0}} \left(\left(p_{N+1-i+j}\right)_{(2(i-j)-1)x}\right)_{\mid \text{all 0}} = s_i \cdot \left(\frac{c_{2N+1}}{w_N^2}\right)^{i},
	\end{align*}
	where the numbers $s_i$ satisfy the recursion
	\begin{align}
		s_i = - \sum_{j=0}^{i-1} \binom{2i-1}{2j} s_j \cdot (-1)^{i-j} t_{i-j} \label{eq:RecursionForsi}
	\end{align}
	and $s_0=1$. We solve this recursion in Lemma~\ref{lem:solutionRecursionRelation} in the appendix and obtain $s_i = \begin{cases}
		1 & i=0 \\
		2^{i-1}  & i>0 \\
	\end{cases}$. This gives the statement of the lemma.
\end{proof}

The previous two lemmas give us control over the $x$-derivatives of $w_i$ at $|_{\mathrm{all}\ 0}$. We now turn to the entries of the Jacobian matrix, which require understanding the partial derivatives of $(w_i)_{jx}$ with respect to the variables $p_N, {w}_{1},{p}_{N-1},{w}_{2},$ ${p}_{N-2},\ldots$, ${w}_{N-1},{p}_{1}, w_N$. For an illustrative example with $N=6$ we refer to the Appendix~\ref{sec:SomeMapleCalculations}.

We begin with the following lemma proving that certain entries are zero (these entries are 
circled in the illustrative examples). Intuitively, this lemma asserts that the partial 
derivative of low-order $x$-derivatives of $w_i$ (or $p_{N+1-i}$) with respect to low degree 
variables vanish at $|_{\mathrm{all}\ 0}$. In particular, specializing to $w_N$, this lemma 
implies that a majority of entries (visually: a staircase) in the strictly upper-left triangular part of the Jacobian $J$ \eqref{ex:AntitriangularStructure} vanishes. Additional work (see \cref{lem:lemma9}) has to be done in order to establish vanishing in this entire region.

\begin{lemma} \label{lem:FirstZeros}
	For each $i \in \{1, \ldots, N\}$ we have for each variable $v \in \{w_k, p_{N+1-k}\}$ of degree $k \in \{1, \ldots, N\}$ and each $j \in \{0,1, \ldots, 2N-2\}$:
	\begin{itemize}
		\item If $i>k+\frac{j}{2}$, then $\left(\frac{\partial (w_i)_{jx}}{\partial v}\right)_{\mid \text{all 0}} = 0$.
		\item If $i>k+\frac{j}{2}$, then $\left(\frac{\partial (p_{N+1-i})_{jx}}{\partial v}\right)_{\mid \text{all 0}} = 0$.
	\end{itemize}

\end{lemma}

\begin{proof} \quad \\
	As before, we proceed by induction on $i$. The base case $i=1$ holds vacuously. 
    
	So let us assume that the statement is true for all indices $i_0<i$ and prove it for $i$. First of all, the case $j=0$ is trivial, as then $i>k$ so $w_i \neq v \neq p_{N+1-i}$. So assume $j\geq 1$. We consider the first derivative of $w_i$:
	\begin{align*}
		(w_i)_{x} &= -\sum_{a=0}^{i-1} w_a p_{N+1-(i-a)}.
	\end{align*}
	Now we differentiate the right-hand side $j-1$ times with respect to $x$ and then differentiate it with respect to $v$. Using the product rule, we obtain:
	\begin{align}
		&\left(\frac{\partial (w_i)_{jx}}{\partial v}\right)_{\mid \text{all 0}} \notag\\ & = - \sum_{a=0}^{i-1} \sum_{b=0}^{j-1} \binom{j-1}{b} \left(\frac{\partial (w_a)_{b\cdot x}}{\partial v} \left(p_{N+1-(i-a)}\right)_{(j-1-b) \cdot x} + (w_a)_{b\cdot x} \frac{\partial \left(p_{N+1-(i-a)}\right)_{(j-1-b) \cdot x}}{\partial v}\right)_{\mid \text{all 0}}. \label{eq:ImportantDerivative}
	\end{align}
	Note that all variables that we differentiate have smaller degree than $i$, so the induction hypothesis applies. Additionally, we can use Lemma~\ref{lem:simpleDerivativesAtAll0} for the terms that are not differentiated with respect to $v$. Let us now consider the terms in each summand individually, starting with the first term $\left(\frac{\partial (w_a)_{b\cdot x}}{\partial v} \left(p_{N+1-(i-a)}\right)_{(j-1-b) \cdot x}\right)_{\mid \text{all 0}}$:
	\begin{itemize}
		\item For $\left(\left(p_{N+1-(i-a)}\right)_{(j-1-b) \cdot x}\right)_{\mid \text{all 0}}$, by Lemma~\ref{lem:simpleDerivativesAtAll0} we know that it vanishes if not $j-1-b \geq 2\cdot (i-a)-1$. We can solve this condition for $a$ and obtain $a \geq i - \frac{j}{2} + \frac{b}{2}$.
		\item For $\left(\frac{\partial (w_a)_{b\cdot x}}{\partial v} \right)_{\mid \text{all 0}}$ not to vanish we must have $a \leq k + \frac{b}{2}$. But $k<i-\frac{j}{2}$, so $a \leq k + \frac{b}{2} < i - \frac{j}{2} + \frac{b}{2}$ has to hold.
	\end{itemize}
	We see that both conditions on the indices cannot be simultaneously satisfied, so in the first term at least one factor always vanishes. So there are no contributions coming from the first term. 
    
	Now we do the same analysis for the second term $\left( (w_a)_{b\cdot x} \frac{\partial \left(p_{N+1-(i-a)}\right)_{(j-1-b) \cdot x}}{\partial v}\right)_{\mid \text{all 0}}$:
	\begin{itemize}
		\item The first factor can be nonzero only if $b \geq 2a$.
		\item The second factor can be nonzero (by induction hypothesis\footnote{Note that for $a=0=b$ we can formally not apply the induction hypothesis, as then we consider $p_{N+1-i}$. In this case, however, we have increased $j$ by 1, so we could do a second induction on $j$ and get the claim from this. We will not repeat this argument for the boundary case.}) only if $i-a \leq k + \frac{j-1-b}{2}$. Solving this for $b$ implies $b \leq 2a +2(\underbrace{k+\frac{j-1}{2}-i}_{<0})$.
	\end{itemize}
	So again, both conditions cannot be simultaneously satisfied and there is no contribution from the second term. We conclude that $\left(\frac{\partial (w_i)_{jx}}{\partial v}\right)_{\mid \text{all 0}} = 0$ as claimed. 
    
	Next, we show that  $\left(\frac{\partial (p_{N+1-i})_{jx}}{\partial v}\right)_{\mid \text{all 0}}$ vanishes. As above, only $j>0$ is relevant. So we consider first
	\begin{align*}
		(p_{N+1-i})_x &= -2w_i + \frac{1}{2} \sum_{a=1}^{i-1} p_{N+1-a} p_{N+1-(i-a)} - \frac{\partial R}{\partial w_{N+1-i}}.
	\end{align*}
    
    First of all, $\left(\frac{\partial (w_i)_{(j-1)x}}{\partial v}\right)_{\mid \text{all 0}} = 0$ (again, a formal proof proceeds by induction on $j$). Now $\frac{\partial R}{\partial w_{N+1-i}}$ is a homogeneous polynomial of degree $i$ containing only $w_1, \ldots, w_{i-1}$, so the induction hypothesis applies. We consider only a monomial $w_{d_1} \cdot \ldots \cdot w_{d_l}$ with $l \geq 2$ and $\sum_{u=1}^{l} d_u = i$. Then
	\begin{align*}
		&\left(\frac{\partial (w_{d_1} \cdot \ldots \cdot w_{d_l})_{(j-1) \cdot x}}{\partial v}\right)_{\mid \text{all 0}} = \\ & = \sum_{b_1+\ldots b_l = j-1} \binom{i}{b_1,\ldots, b_l} \left( \frac{\partial (w_{d_1})_{b_1 \cdot x}}{\partial v} \cdot (w_{d_2})_{b_2\cdot x} \ldots (w_{d_l})_{b_l \cdot x} + \ldots + (w_{d_1})_{b_1\cdot x} \ldots (w_{d_{l-1}})_{b_{l-1} \cdot x} \cdot \frac{\partial (w_{d_l})_{b_l \cdot x}}{\partial v} \right)_{\mid \text{all 0}}.
	\end{align*}
	Next, we consider only $\left( \frac{\partial (w_{d_1})_{b_1 \cdot x}}{\partial v} \cdot (w_{d_2})_{b_2\cdot x} \ldots (w_{d_l})_{b_l \cdot x}\right)_{\mid \text{all 0}}$, the remaining terms follow by the same argument.
	\begin{itemize}
		\item For $r>2$, in order for $(w_{d_r})_{b_r\cdot x}$ not to be zero we must have $b_r \geq 2 d_r$. So in total, \linebreak $j-1-b_1 = \sum_{r=2}^{l} b_r \geq 2 \sum_{r=2}^{l} d_r = 2 (i-d_1)$. So $d_1 \geq i - \frac{j-1}{2} + \frac{b_1}{2}$.
		\item For $ \frac{\partial (w_{d_1})_{b_1 \cdot x}}{\partial v}$ not to vanish, we need $d_1 \leq k + \frac{b_1}{2} < i- \frac{j}{2} + \frac{b_1}{2}$. 
	\end{itemize}
	Both conditions cannot be satisfied, so this term vanishes. Thus, we have proved that
	\begin{align}
		&\left(\frac{\partial (p_{N+1-i})_{jx}}{\partial v}\right)_{\mid \text{all 0}} = \frac{1}{2} \left(\frac{\partial \left(\sum\limits_{a=1}^{i-1} p_{N+1-a} p_{N+1-(i-a)} \right)_{(j-1) \cdot x}}{\partial v}\right)_{\mid \text{all 0}} = \notag \\
		&= \sum\limits_{a=1}^{i-1} \sum_{b=0}^{j-1} \binom{j-1}{b} \left( \frac{\partial \left(p_{N+1-a}\right)_{b \cdot x}}{\partial v} \left(p_{N+1-(i-a)}\right)_{(j-1-b) \cdot x}\right)_{\mid \text{all 0}}. \label{eq:DerivativeP}
	\end{align}
	Note that we used symmetry to only get one term, but it eliminated the factor of $\frac{1}{2}$. We now analyze the indices of the factors to show vanishing of their product:
	\begin{itemize}
		\item For the second factor, we need $j-1-b \geq 2(i-a) -1$, so $a \geq i - \frac{j}{2} + \frac{b}{2} > k + \frac{b}{2}$.
		\item For the first factor, we need $a\leq k + \frac{b}{2}$.
	\end{itemize}
	These two conditions cannot be satisfied simultaneously and so these their products vanish. This 
    proves the induction step and we conclude the claim.
\end{proof}


The purpose of the following lemma is two-fold: it establishes that the remaining strictly 
upper-left triangular entries of $J$ vanish, and shows non-vanishing of some entries on the anti-diagonal.
\Cref{lem:lemma10} will establish non-vanishing of the remaining 
entries on the anti-diagonal.  In \cref{ex:AntitriangularStructure} 
and the appendix we illustrate this for $N = 6$, with these entries 
indicated by a dashed circle.

\begin{lemma}\label{lem:lemma9}
	For each $i \in \{1, \ldots, N\}$ we have for each variable $v \in \{w_k, p_{N+1-k}\}$ of degree $k \in \{1, \ldots, N\}$ and each $j \in \{0,1, \ldots, 2N-2\}$ with $i=k+\frac{j}{2}$:
	\begin{itemize}
		\item For $v= w_k$ we have $\left(\frac{\partial (w_i)_{jx}}{\partial v}\right)_{\mid \text{all 0}} = s_{\frac{j}{2}} \left(\frac{c_{2N+1}}{w_N^2}\right)^{\frac{j}{2}}$, where  $s_l = \begin{cases} 2^{l-1} & \quad l>0 \\ 1, & \quad l=0 \end{cases}$. Note that $\frac{j}{2} = i-k$, so we could also write $i-k$ at any occurrence of $\frac{j}{2}$.
		\item For $v= p_{N+1-k}$ we have $\left(\frac{\partial (w_i)_{jx}}{\partial v}\right)_{\mid \text{all 0}} = 0$.
		\item For $v= w_{k}$ we have $\left(\frac{\partial (p_{N+1-i})_{jx}}{\partial v}\right)_{\mid \text{all 0}} = 0$.
		\item For $v= p_{N+1-k}$ we have $\left(\frac{\partial (p_{N+1-i})_{jx}}{\partial v}\right)_{\mid \text{all 0}} = (-1)^{\frac{j}{2}} \cdot t_{\frac{j}{2}+1} \left(\frac{c_{2N+1}}{w_N^2}\right)^{\frac{j}{2}}$, where $t_l$ are the reduced tangent numbers (see Lemma~\ref{lem:RecursionReducedTangentNumbers}).
	\end{itemize}
\end{lemma}

\begin{proof}
	\quad \\
	The proof is by induction on $i$ and similar to the proof of Lemma~\ref{lem:FirstZeros}. Note also that $j$ has to be even, as otherwise $i=k+\frac{j}{2}$ is impossible. The case $j=0$ is always trivial, as then $i=k$ and we have no $x$-derivatives. This is also the only possible case for $i=1$, so the base case is proven. \\
	Consider for the inductive step some index $i>1$ and assume that the statement is proven for all $i_0<i$. Assume $j>0$, so we have at least one derivative. First, we focus on $w_i$. Then we can use \eqref{eq:ImportantDerivative}, which we repeat for convenience:
	\begin{align*}
		&\left(\frac{\partial (w_i)_{jx}}{\partial v}\right)_{\mid \text{all 0}} \notag\\ & = - \sum_{a=0}^{i-1} \sum_{b=0}^{j-1} \binom{j-1}{b} \left(\frac{\partial (w_a)_{b\cdot x}}{\partial v} \left(p_{N+1-(i-a)}\right)_{(j-1-b) \cdot x} + (w_a)_{b\cdot x} \frac{\partial \left(p_{N+1-(i-a)}\right)_{(j-1-b) \cdot x}}{\partial v}\right)_{\mid \text{all 0}}.
	\end{align*}
	Now we consider the $v=p_{N+1-k}$ case first. Consider the first term $\left(\frac{\partial (w_a)_{b\cdot x}}{\partial v} \left(p_{N+1-(i-a)}\right)_{(j-1-b) \cdot x}\right)_{\mid \text{all 0}}$:
	\begin{itemize}
		\item By the induction hypothesis, $\left(\frac{\partial (w_a)_{b\cdot x}}{\partial v} \right)_{\mid \text{all 0}}$ can only be nonzero if $a<k+\frac{b}{2}$ (the case of equality is not possible).
		\item For $\left(\left(p_{N+1-(i-a)}\right)_{(j-1-b) \cdot x}\right)_{\mid \text{all 0}}$ to be nonzero, we need $j-1-b \geq 2(i-a)-1$, or equivalently $a \geq i-\frac{j}{2} + \frac{b}{2} = k+\frac{b}{2}$. 
	\end{itemize}
	We see that both conditions cannot be satisfied simultaneously, so the first term vanishes. 
    
	For the second term $\left((w_a)_{b\cdot x} \frac{\partial \left(p_{N+1-(i-a)}\right)_{(j-1-b) \cdot x}}{\partial v}\right)_{\mid \text{all 0}}$ we obtain:
	\begin{itemize}
		\item For $\left((w_a)_{b\cdot x}\right)_{\mid \text{all 0}}$ we need $a \leq \frac{b}{2}$.
		\item For $\left(\frac{\partial \left(p_{N+1-(i-a)}\right)_{(j-1-b) \cdot x}}{\partial v}\right)_{\mid \text{all 0}}$ we need (note that now we differentiate with respect to a momentum, so equality would yield a non-zero value) $i-a \leq k + \frac{j-1-b}{2}$ or $a\geq \frac{b+1}{2}$ (using $i=k+\frac{j}{2}$). 
	\end{itemize}
	So again, all products vanish and hence do not contribute. In total, there is no contribution and so $\left(\frac{\partial (w_i)_{jx}}{\partial p_{N+1-k}}\right)_{\mid \text{all 0}} = 0$ as claimed. 
    
	Next, we consider the case $v=w_k$. The analysis is similar as above, but at one place we switch from a strict inequality to an equality that allows non-zero terms. In detail, consider first the second summand $\left((w_a)_{b\cdot x} \frac{\partial \left(p_{N+1-(i-a)}\right)_{(j-1-b) \cdot x}}{\partial v}\right)_{\mid \text{all 0}}$
	\begin{itemize}
		\item For $\left((w_a)_{b\cdot x}\right)_{\mid \text{all 0}}$ we need $a \leq \frac{b}{2}$.
		\item For $\left(\frac{\partial \left(p_{N+1-(i-a)}\right)_{(j-1-b) \cdot x}}{\partial v}\right)_{\mid \text{all 0}}$ we need  $i-a < k + \frac{j-1-b}{2}$ or $a> \frac{b+1}{2}$ (using $i=k+\frac{j}{2}$). 
	\end{itemize}
	So the second summands vanish. 
    
	But for $\left(\frac{\partial (w_a)_{b\cdot x}}{\partial v} \left(p_{N+1-(i-a)}\right)_{(j-1-b) \cdot x}\right)_{\mid \text{all 0}}$ we obtain:
	\begin{itemize}
		\item By the induction hypothesis, $\left(\frac{\partial (w_a)_{b\cdot x}}{\partial v} \right)_{\mid \text{all 0}}$ can only be nonzero if $a\leq k+\frac{b}{2}$ (the case of equality is now possible).
		\item For $\left(\left(p_{N+1-(i-a)}\right)_{(j-1-b) \cdot x}\right)_{\mid \text{all 0}}$ to be nonzero, we need $j-1-b \geq 2(i-a)-1$, or equivalently $a \geq i-\frac{j}{2} + \frac{b}{2} = k+\frac{b}{2}$. 
	\end{itemize}
	So now there is a contribution for $a=k+\frac{b}{2}$ or $b=2(a-k)$. We simplify the expression:
	\begin{align*}
		\left(\frac{\partial (w_i)_{jx}}{\partial v}\right)_{\mid \text{all 0}} &= - \sum_{a=0}^{i-1} \binom{j-1}{2(a-k)} \left(\frac{\partial (w_a)_{(2(a-k))\cdot x}}{\partial v} \left(p_{N+1-(i-a)}\right)_{(j-1-2(a-k)) \cdot x}\right)_{\mid \text{all 0}} = \\
		& \stackrel{\text{IH}}{=} - \sum_{a=0}^{i-1} \binom{j-1}{2(a-k)} s_{a-k} \left(\frac{c_{2N+1}}{w_N^2}\right)^{a-k}  (-1)^{i-a} t_{i-a} \left(\frac{c_{2N+1}}{w_N^2}\right)^{i-a} = \\
		& = \left(- \sum_{l=0}^{i-k-1} \binom{2(i-k)-1}{2 l} s_{l}   (-1)^{i-k-l} t_{i-k-l}\right) \cdot \left(\frac{c_{2N+1}}{w_N^2}\right)^{i-k}.
	\end{align*}
	Here we have used the relation $j = 2 (i-k)$ to eliminate $j$ from the expression. It remains to argue that the first factor equals $s_{i-k}$. In general, we need the recursion relation
	\begin{align*}
		- \sum_{l=0}^{b-1} \binom{2b-1}{2 l} s_{l}   (-1)^{b-l} t_{b-l} = s_b
	\end{align*}
	for any $b>0$, and this is derived in Lemma~\ref{lem:solutionRecursionRelation}, see also \eqref{eq:RecursionForsi}. So the proof for $w_i$ is completed. \\
	Next, we consider $p_{N+1-i}$. We have again
	\begin{align*}
		(p_{N+1-i})_x &= -2w_i + \frac{1}{2} \sum_{a=1}^{i-1} p_{N+1-a} p_{N+1-(i-a)} - \frac{\partial R}{\partial w_{N+1-i}}.
	\end{align*}
	The first term $-2w_i$ will not yield a contribution by Lemma~\ref{lem:FirstZeros}. We have to argue that this is also the case for $\frac{\partial R}{\partial w_{N+1-i}}$. For concreteness, we sketch the argument for a term $w_{\alpha} \cdot w_{i-\alpha}$ with $1\leq \alpha <i$, the general argument can be derived from this (similar as in Lemma~\ref{lem:FirstZeros}):
	\begin{align*}
		\left(\frac{\partial \left(w_{\alpha} \cdot w_{i-\alpha}\right)_{(j-1)\cdot x}}{\partial v}\right)_{\mid \text{all 0}} &= \sum_{b=0}^{j-1} \binom{j-1}{b} \left(\frac{\partial (w_{\alpha})_{b\cdot x}}{\partial v} (w_{i-\alpha})_{(j-1-b)\cdot x}+  (w_{\alpha})_{b\cdot x} \cdot \frac{\partial (w_{i-\alpha})_{(j-1-b)\cdot x}}{\partial v}\right)_{\mid \text{all 0}}.
	\end{align*}
	We focus on the first term, the second term being similar.
	\begin{itemize}
		\item For $\left(\frac{\partial (w_{\alpha})_{b\cdot x}}{\partial v} \right)_{\mid \text{all 0}}$ to be nonzero, we need (according to Lemma~\ref{lem:FirstZeros}, i.e. we do not distinguish cases on the variable $v$) $\alpha \leq k + \frac{b}{2}$.
		\item For $\left( (w_{i-\alpha})_{(j-1-b)\cdot x}\right)_{\mid \text{all 0}}$ to be nonzero, we need $i-\alpha \leq \frac{(j-1-b)}{2}$, or $\alpha \geq i-\frac{j-1}{2} + \frac{b}{2} = k +\frac{1}{2} + \frac{b}{2}$.
	\end{itemize}
	Both cannot be simultaneously satisfied, so these terms do not contribute. 
    
	So finally, we only have to investigate (compare \eqref{eq:DerivativeP})
	\begin{align*}
		&\left(\frac{\partial (p_{N+1-i})_{jx}}{\partial v}\right)_{\mid \text{all 0}} = \frac{1}{2} \left(\frac{\partial \left(\sum\limits_{a=1}^{i-1} p_{N+1-a} p_{N+1-(i-a)} \right)_{(j-1) \cdot x}}{\partial v}\right)_{\mid \text{all 0}} =  \\
		&= \sum\limits_{a=1}^{i-1} \sum_{b=0}^{j-1} \binom{j-1}{b} \left( \frac{\partial \left(p_{N+1-a}\right)_{b \cdot x}}{\partial v} \left(p_{N+1-(i-a)}\right)_{(j-1-b) \cdot x}\right)_{\mid \text{all 0}}.
	\end{align*}
	Consider again the case $v=w_k$ first, as this should yield zero.
	\begin{itemize}
		\item For $\left( \frac{\partial \left(p_{N+1-a}\right)_{b \cdot x}}{\partial v}\right)_{\mid \text{all 0}}$ not to vanish we need $a<k+\frac{b}{2}$.
		\item For $\left( \left(p_{N+1-(i-a)}\right)_{(j-1-b) \cdot x}\right)_{\mid \text{all 0}}$ not to vanish we need $j-1-b \geq 2(i-a)-1$ or $a \leq i-\frac{j}{2} + \frac{b}{2} = k + \frac{b}{2}$. 
	\end{itemize}
	Both conditions cannot be satisfied simultaneously, so the case $v=w_k$ is proven. 
    
	For $v=p_{N+1-k}$, however, the analysis is as follows:
	\begin{itemize}
		\item For $\left( \frac{\partial \left(p_{N+1-a}\right)_{b \cdot x}}{\partial v}\right)_{\mid \text{all 0}}$ not to vanish we need $a\leq k+\frac{b}{2}$.
		\item For $\left( \left(p_{N+1-(i-a)}\right)_{(j-1-b) \cdot x}\right)_{\mid \text{all 0}}$ not to vanish we need $j-1-b \geq 2(i-a)-1$ or $a \leq i-\frac{j}{2} + \frac{b}{2} = k + \frac{b}{2}$. 
	\end{itemize}
	So we get a contribution for $b=2(a-k)$. Recall that $j= 2(i-k)$ and plug this into the formula:
	\begin{align*}
		\left(\frac{\partial (p_{N+1-i})_{jx}}{\partial v}\right)_{\mid \text{all 0}} &= \sum\limits_{a=1}^{i-1} \sum_{b=0}^{j-1} \binom{j-1}{b} \left( \frac{\partial \left(p_{N+1-a}\right)_{b \cdot x}}{\partial v} \left(p_{N+1-(i-a)}\right)_{(j-1-b) \cdot x}\right)_{\mid \text{all 0}} = \\
		&= \sum\limits_{a=k}^{i-1} \binom{2(i-k)-1}{2(a-k)} \left( \frac{\partial \left(p_{N+1-a}\right)_{2(a-k) \cdot x}}{\partial v} \left(p_{N+1-(i-a)}\right)_{(2(i-a)-1) \cdot x}\right)_{\mid \text{all 0}} = \\
		&= \sum\limits_{l=0}^{i-k-1} \binom{2(i-k)-1}{2l} 
		\left( \frac{\partial \left(p_{N+1-(k+l)}\right)_{2l \cdot x}}{\partial v} \left(p_{N+1-(i-k-l)}\right)_{(2(i-k-l)-1) \cdot x}\right)_{\mid \text{all 0}} = \\
		&= \sum\limits_{l=0}^{i-k-1} \binom{2(i-k)-1}{2l} (-1)^{l} t_{l+1} \left(\frac{c_{2N+1}}{w_N^2}\right)^{l} (-1)^{i-k-l} t_{i-k-l} \left(\frac{c_{2N+1}}{w_N^2}\right)^{i-k-l} = \\
		&= (-1)^{i-k} \left(\frac{c_{2N+1}}{w_N^2}\right)^{i-k} \quad \sum\limits_{l=0}^{i-k-1} \binom{2(i-k)-1}{2l} t_{l+1} t_{i-k-l}.
	\end{align*}
    The result now follows from the recursion
    \begin{align*}
    t_{b+1} = \sum_{l=0}^{b-1} \binom{2b-1}{2l} t_{l+1} t_{b-l} = \sum_{l=1}^{b} \binom{2b-1}{2(l-1)} t_l t_{b+1-l},
    \end{align*}
    for any $b \in \mathbb{N}$, which is established in 
    \cref{lem:AlternativeRecursionReducedTangentNumbers} in the appendix. This completes the proof.
\end{proof}

The last step for proving our main theorem on the Camassa--Holm 
equation is the following lemma (considering the other anti-diagonal 
entries for $j$ odd). In \cref{ex:AntitriangularStructure} and the 
appendix we illustrate this for $N = 6$, with these entries indicated 
by a dashed box around them.

\begin{lemma}\label{lem:lemma10} \quad \\
	For each $i \in \{1, \ldots, N\}$ we have for each $k \in \{1, \ldots, N\}$ and each $j \in \{0,1, \ldots, 2N-2\}$ with $i=k+\frac{j-1}{2}$:
	\begin{align*}
		\left(\frac{\partial (w_i)_{jx}}{\partial p_{N+1-k}}\right)_{\mid \text{all 0}} = - 2^{i-k} \left(\frac{c_{2N+1}}{w_N^2}\right)^{i-k}.
	\end{align*}
\end{lemma}

\begin{proof}
	\quad \\
	The proof is via induction on $i$. We directly start with the inductive step (the base case can be treated the same way). Again, we use \eqref{eq:ImportantDerivative} from Lemma~\ref{lem:FirstZeros} and write $v$ for $p_{N+1-k}$ for brevity:
	\begin{align*}
		&\left(\frac{\partial (w_i)_{jx}}{\partial v}\right)_{\mid \text{all 0}} \\ & = - \sum_{a=0}^{i-1} \sum_{b=0}^{j-1} \binom{j-1}{b} \left(\frac{\partial (w_a)_{b\cdot x}}{\partial v} \left(p_{N+1-(i-a)}\right)_{(j-1-b) \cdot x} + (w_a)_{b\cdot x} \frac{\partial \left(p_{N+1-(i-a)}\right)_{(j-1-b) \cdot x}}{\partial v}\right)_{\mid \text{all 0}}.
	\end{align*}
	Now we consider again the summands, starting with $\left(\frac{\partial (w_a)_{b\cdot x}}{\partial v} \left(p_{N+1-(i-a)}\right)_{(j-1-b) \cdot x}\right)_{\mid \text{all 0}}$:
	\begin{itemize}
		\item For the first factor $\left(\frac{\partial (w_a)_{b\cdot x}}{\partial v}\right)_{\mid \text{all 0}}$ not to vanish we need $a<k+ \frac{b}{2}$.
		\item For the second factor $\left(\left(p_{N+1-(i-a)}\right)_{(j-1-b) \cdot x}\right)_{\mid \text{all 0}}$ we need $j-1-b \geq 2(i-a)-1$, or equivalently $a \geq i-\frac{j}{2} + \frac{b}{2} = k - \frac{1}{2} + \frac{b}{2}$.
	\end{itemize}
	This is possible only for $b=2(a-k)+1$. So from the first terms in the summands we get the contribution
	\begin{align*}
		&- \sum_{a=0}^{i-1} \sum_{b=0}^{j-1} \binom{j-1}{b} \left(\frac{\partial (w_a)_{b\cdot x}}{\partial v} \left(p_{N+1-(i-a)}\right)_{(j-1-b) \cdot x}\right)_{\mid \text{all 0}} = \\ & = - \sum_{a=k}^{i-1} \binom{j-1}{2(a-k)+1} \left(\frac{\partial (w_a)_{(2(a-k)+1)\cdot x}}{\partial v} \left(p_{N+1-(i-a)}\right)_{(j-2(a-k)-2) \cdot x}\right)_{\mid \text{all 0}}.
	\end{align*}
	Note that $2i=2k+j-1$, so $j-2(a-k)-2 = 2(i-a)-1$. Now we can apply the induction hypothesis and the above lemmas to obtain:
	\begin{align*}
		&= \sum_{a=k}^{i-1} \binom{j-1}{2(a-k)+1} 2^{a-k} \left(\frac{c_{2N+1}}{w_N^2}\right)^{a-k} \cdot (-1)^{i-a} t_{i-a} \left(\frac{c_{2N+1}}{w_N^2}\right)^{i-a} = \\& = 
		(-1)^{i-k} \left(\frac{c_{2N+1}}{w_N^2}\right)^{i-k} \quad \sum_{l=0}^{i-k-1} \binom{2(i-k)}{2l+1} 2^{l}  \cdot (-1)^{l} t_{i-k-l}.
	\end{align*}
	Next, consider the second term $\left((w_a)_{b\cdot x} \frac{\partial \left(p_{N+1-(i-a)}\right)_{(j-1-b) \cdot x}}{\partial v}\right)_{\mid \text{all 0}}$
	\begin{itemize}
		\item For the first factor  $\left((w_a)_{b\cdot x} \right)_{\mid \text{all 0}}$ not to vanish we need $b \geq 2a$.
		\item For the second factor $\left(\frac{\partial \left(p_{N+1-(i-a)}\right)_{(j-1-b) \cdot x}}{\partial v}\right)_{\mid \text{all 0}}$ not to vanish we need $i-a \leq k + \frac{j-1-b}{2}$ or equivalently $b \leq 2a + 2 \left(k+\frac{j-1}{2} -i\right) = 2a$.
	\end{itemize}
	So again there is one contribution for $b=2a$. Thus, the sum of the second terms is
	\begin{align*}
		& - \sum_{a=0}^{i-1} \sum_{b=0}^{j-1} \binom{j-1}{b} \left((w_a)_{b\cdot x} \frac{\partial \left(p_{N+1-(i-a)}\right)_{(j-1-b) \cdot x}}{\partial v}\right)_{\mid \text{all 0}} = \\
		& = - \sum_{a=0}^{i-1} \binom{2(i-k)}{2a} \left((w_a)_{(2a)\cdot x} \frac{\partial \left(p_{N+1-(i-a)}\right)_{(j-1-2a) \cdot x}}{\partial v}\right)_{\mid \text{all 0}} = \\
		&= - \sum_{a=0}^{i-1} \binom{2(i-k)}{2a} s_a \left(\frac{c_{2N+1}}{w_N^2}\right)^{a} (-1)^{i-k-a} t_{i-k-a+1} \left(\frac{c_{2N+1}}{w_N^2}\right)^{i-k-a} = \\
		&= (-1)^{i-k+1} \left(\frac{c_{2N+1}}{w_N^2}\right)^{i-k} \quad  \sum_{a=0}^{i-k} \binom{2(i-k)}{2a} s_a (-1)^{a} t_{i-k-a+1}.
	\end{align*}
	Putting everything together, we obtain:
	\begin{align*}
		\left(\frac{\partial (w_i)_{jx}}{\partial v}\right)_{\mid \text{all 0}} &= (-1)^{i-k} \left(\frac{c_{2N+1}}{w_N^2}\right)^{i-k} \quad \sum_{l=0}^{i-k-1} \binom{2(i-k)}{2l+1} s_{l+1}  \cdot (-1)^{l} t_{i-k-l} + \\ & \quad + (-1)^{i-k+1} \left(\frac{c_{2N+1}}{w_N^2}\right)^{i-k} \quad  \sum_{a=0}^{i-k} \binom{2(i-k)}{2a} s_a (-1)^{a} t_{i-k-a+1} = \\
		& = (-1) \cdot (-1)^{i-k} \left(\frac{c_{2N+1}}{w_N^2}\right)^{i-k} \cdot \\ & \cdot \left(\sum_{a=0}^{i-k} \binom{2(i-k)}{2a} s_a (-1)^{a} t_{i-k-a+1} - \sum_{a=1}^{i-k} \binom{2(i-k)}{2a-1} s_{a}  \cdot (-1)^{a-1} t_{i-k-a+1}\right) = \\
		& = (-1) \cdot (-1)^{i-k} \left(\frac{c_{2N+1}}{w_N^2}\right)^{i-k} \cdot \left(\sum_{a=0}^{i-k} \left(\binom{2(i-k)}{2a}+ \binom{2(i-k)}{2a-1}\right) s_a (-1)^{a} t_{i-k-a+1} \right) = \\
		&=  (-1) \cdot (-1)^{i-k} \left(\frac{c_{2N+1}}{w_N^2}\right)^{i-k} \cdot \left(\sum_{a=0}^{i-k} \binom{2(i-k)+1}{2a} s_a (-1)^{a} t_{i-k-a+1} \right).
	\end{align*}
	So we need to prove that $\sum_{a=0}^{i-k} \binom{2(i-k)+1}{2a} s_a (-1)^{a} t_{i-k-a+1} = (-1)^{i-k} s_{i-k+1}$, then we are done. But this is the same as \eqref{eq:RecursionForsi}. Set $b=i-k+1$, then
	\begin{align*}
		\sum_{a=0}^{i-k} \binom{2(i-k)+1}{2a} s_a (-1)^{a} t_{i-k-a+1} &= (-1)^{i-k+1} \cdot \sum_{a=0}^{i-k} \binom{2(i-k)+1}{2a} s_a (-1)^{i-k-a+1} t_{i-k-a+1} = \\
		&= (-1)^{i-k+1} \cdot \sum_{a=0}^{b-1} \binom{2b-1}{2a} s_a (-1)^{b-a} t_{b-a} = (-1)^{i-k+1} (-s_b) = \\ & = (-1)^{i-k} s_{i-k+1}.
	\end{align*}
	So everything is proven.
\end{proof}

To summarize, \cref{prop:structure_jacobian} follows from lemmas~\ref{lem:FirstZeros}--
\ref{lem:lemma10} by specializing to the case $w_N$, and \cref{prop:structure_jacobian} implies 
the theorem for the special case $m(\mu) = \mu$. In the next section we extend this result to all 
polynomials $m(\mu) = m_1 \mu + m_0$, so that the theorem is proven in full generality.

\subsection{Generalization to the case of arbitrary polynomial for $n=1$} \label{subsec:GeneralizationArbitraryPolynomial}

We now reduce the jet-approximation problem for a general degree $1$ polynomial $m(\mu) = m_1 \mu 
+ m_0$ to the Camassa--Holm case $m(\mu) = \mu$. In \cite{konyaev2025lax}, the authors mention a 
gauge equivalence on the class of BKM systems which takes the form 
\begin{equation}
    \lambda \mapsto \lambda + \lambda_0, \qquad m(\mu) \mapsto m(\mu - \lambda_0).
\end{equation}
Clearly, the polynomial $m(\lambda)$ is preserved. In view of \eqref{eq:BKM_typeI}, it transforms the BKM 
system to one of the same type (using that adding a multiple of the 
identity to a Nijenhuis operator is still Nijenhuis). Using this gauge 
transformation, we can eliminate the constant term from $m(\mu) = m_1 
\mu + m_0$. Thus, we may assume that $m(\mu) = m_1 \mu$ with $m_1 \neq 
0$.

We show first that the theorem holds for all $m_1 > 0$. To this end, 
consider the following transformation parametrized by a constant $a \neq 0$:
\begin{equation}
    x \mapsto a\,x,\quad t \mapsto a\,t,\quad m(\mu) \mapsto a^2\, m(\mu).
\end{equation}
In particular, this action induces $\zeta \mapsto a^2 \zeta$ and we see that this action 
preserves the class of BKM systems of a given type \eqref{eq:BKM_typeI}. By using this $\mathbb{R}_{\times}$-action we can rescale $m_1$ to be either $\pm 1$ (over $\mathbb{C}$ to 1). The claim 
is true for $m_1 = 1$ and so we obtain the theorem for all $m_1 > 0$. Now, our jet-approximation approach depends algebraically on the data. Therefore, the result must be true for a Zariski-open set containing $\{m_1 : m_1 > 0 \} $, which implies that it is true for all $m_1 \neq 0$. This concludes the proof of \cref{thm:ApproximationCasemlinear}.

The above approach relied on transformations of the configuration space of the BKM PDE. An 
alternative approach is to instead consider a transformation on the configuration space $
\mathbb{R}^N(w)$ of the Stäckel system, which we briefly sketch here. \\
When working in the coordinates $w$, the constant term of $c(\mu)$ gives a contribution of $\left(-{\sum_{i=0}^{N} 
w_{N-i} \lambda^{i}}\right)^{-1}$. This is no longer a polynomial of homogeneous degree in $w_1, 
\ldots, w_N$, so we cannot rely on techniques used in the previous sections. Define new variables $v_1, \ldots, v_N$ through the relation
\begin{align}
	\mu^{N} + \sum_{i=1}^{N} w_i \mu^{N-i} \stackrel{!}{=} (\mu-\lambda)^{N} + \sum_{i=1}^{N} v_i (\mu-\lambda)^{N-i}. \label{eq:CoordinateChange}
\end{align}
Note that the change of variables from $w$ to $v$ and vice versa is given by an affine map, and 
its inverse can be readily obtained. One can obtain the Nijenhuis operator $M$ and inverse flat metric $g^{-1}$ in terms of $v$-coordinates, as follows.

\begin{lemma}[\textbf{$M, g^{-1}$ in new coordinates}] \quad \\
	The matrix $M_v$ representing the Nijenhuis operator $M$ \eqref{eq:M_and_g} in terms of $v$-coordinates is given by
	\begin{align*}
		M_v &= \begin{pmatrix}
			-v_1 & 1 & 0 & \ldots & 0 \\
			-v_2 & 0 & 1 & \ddots & \vdots \\
			\vdots & & \ddots & \ddots & \\
			-v_{N-1} & & & 0 & 1 \\
			-v_{N} & 0 & & & 0
		\end{pmatrix} + \lambda \cdot \operatorname{Id} = \begin{pmatrix}
			-v_1+\lambda & 1 & 0 & \ldots & 0 \\
			-v_2 & \lambda & 1 & \ddots & \vdots \\
			\vdots & & \ddots & \ddots & \\
			-v_{N-1} & & & \lambda & 1 \\
			-v_{N} & 0 & & & \lambda
		\end{pmatrix}.
	\end{align*}
    The matrix $g_v^{-1}$ representing the inverse metric $g^{-1}$ \eqref{eq:M_and_g} in terms of $v$-coordinates is given by
	\begin{align*}
		g_v^{-1} &= \begin{pmatrix}
			0 & \ldots & 0 & 0  & 1 \\
			0 & \ldots & 0 & 1  & v_1 \\
			\vdots & \iddots  & \iddots & \iddots  & v_2 \\
			0 & 1 &  v_1 & \iddots &  \vdots \\
			1 & v_1 & v_2 & \ldots & v_{N-1}
		\end{pmatrix}.
	\end{align*}
\end{lemma}
We can now transfer the jet-approximation results from the Camassa--Holm case. Consider the first 
Hamiltonian (with flow parameter $x$):
\begin{align*}
	H_{\mathrm{CH}, 1} = K_{\mathrm{CH}, 1}(w, p) + V_{\mathrm{CH}, 1}(w).
\end{align*}
In the $v$-coordinates the kinetic energy is given by $\frac{1}{2} p_v^{T} g_v^{-1} p_v$ (and 
this is independent of $m(\mu)$). We expand the polynomial $c(\mu)$ in the basis $1, \mu-\lambda, \ldots, (\mu-\lambda)^{2N+1}$, i.e.
\begin{align*}
	c(\mu) &= \sum_{i=0}^{2N+1} c_i (\mu-\lambda)^{2N+1-i},
\end{align*}
where $c_0 :=1$. For concreteness, let us consider $c(\mu) = (\mu-\lambda)^{2N+1} + c_{2N+1}$. 
Then we have 
\begin{align*}
	f(\mu) = \frac{c(\mu)}{m(\mu)} = (\mu-\lambda)^{2N} + \frac{c_{2N+1}}{\mu-\lambda}.
\end{align*}
To rewrite $f(M_v)$ as a linear combination of $M_v^{N-1}, \ldots, M_v, \operatorname{Id}$, note first that
\begin{align*}
	f(M_v) =  (M_v-\lambda \cdot \operatorname{Id})^{2N} + c_{2N+1}(M_v-\lambda \cdot \operatorname{Id})^{-1}.
\end{align*}
In view of the lemma, we see that $M_v-\lambda \cdot \operatorname{Id}$ is essentially the companion matrix $M_w$ with 
the variables renamed. By using the potential of Camassa--Holm for $\tilde{c}(\mu) = \mu^{2N+1} + c_{2N+1}$ we can rewrite
\begin{align*}
	f(M_v) &= V_{\mathrm{CH},1}(v) (M_v-\lambda \cdot \operatorname{Id})^{N-1}+ V_{\mathrm{CH},2}(v) (M_v-\lambda \cdot \operatorname{Id})^{N-2} + \ldots + V_{\mathrm{CH}, N} \operatorname{Id}.
\end{align*}
In order to find the potential, we need to write these expressions in terms 
of powers of $M_v$ (and not $M_v-\lambda \cdot \operatorname{Id}$). However, the term in 
$M_v^{N-1}$ can only come from the first summand, so
\begin{align*}
	f(M_v) &= V_{\mathrm{CH},1}(v) M_v^{N-1}+ \ldots M_v^{N-2} + \ldots .
\end{align*}
For the purpose of the jet-approximation problem, we are thus in the same situation as for 
Camassa--Holm, the only difference being that we consider renamed variables. Thus, we can 
approximate $v_N$ up to order $2N-2$. The final step concerns the reduction map $\mathcal{R}$. By definition, we 
require that $\sigma_v(\mu)^2 \sigma_u(\mu) - c(\mu)$ is divisible by $\mu-\lambda$, so
\begin{align*}
	0 \stackrel{!}{=} \sigma_v(\lambda)^2 (u_1-\lambda) - c(\lambda) = v_N^2 (u_1-\lambda) - c_{2N+1}.
\end{align*}
Solving yields $u_{1} = \frac{c_{2N+1}}{v_N^2} + \lambda$ and we conclude that 
approximating $u_1$ is equivalent to approximating $v_N$.

\section{Outlook}
\label{sec:Outlook}

\subsection{Proving jet-density for more general (BKM) systems}

In this paper we showed jet-density of finite-gap solutions for the 
case $\deg(m) = 0$ with arbitrarily many components as well as the 
case with $\deg(m) = 1$ and one component. It would be of interest 
to establish or disprove jet-density for a 
general BKM system. Numerical experiments for $n = 2, m(\mu) = \mu$ 
and $n = 2, m(\mu) = \mu^2$ point towards jet-density holding also 
in these cases. In order to obtain jet-density, we used the 
cotangent coordinates. One can also use parameters 
coming from the spectral curve $c(\mu)$ for this purpose, and one 
could imagine that these are important in more general cases as 
they are not dynamical variables. In principle, one could aim to 
apply these techniques to PDEs that are not BKM, for example, the 
Kadomtsev--Petviashvili (KP) equation in $2 + 1$ dimensions.

\subsection{Convergence of finite-gap solutions}

The approach of convergence of Taylor polynomials up to arbitrary order is often called \enquote{convergence in a formal sense}. It would be interesting to study whether one could get a stronger notion of convergence, e.g. locally uniform convergence in a neighborhood of $u(0)$ (assuming the initial data for the BKM system is analytic). One potential obstacle is the following: The solution of the finite-gap ODE need not be defined for all $x \in \R$, but only on a small interval. By increasing $N$, the radius of convergence could shrink and in the limit tend to $0$. We explored this numerically but found no evidence that this shrinking behavior occurs (see example below). This suggests that locally uniform convergence might indeed be possible.

\begin{example}

	Consider the initial data $u(x) = e^{-x^2}$. By approximating this function up to some order in the case $m\equiv 1, n=1$, we obtain solutions as in Figure~\ref{fig:ExampleetoMinusx2}.

	\begin{figure}[htbp]

		\includegraphics[width=0.5\textwidth]{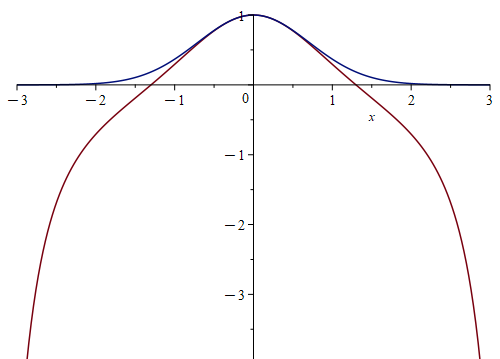}

		\caption{A sketch of the initial data $x\mapsto e^{-x^2}$ (blue) and the function solving the BKM system approximating it up to order 3 in $x_0=0$ using solutions from an ODE solver (red). For higher orders, the picture looks similar.}
		\label{fig:ExampleetoMinusx2}

	\end{figure}

\end{example}

These functions being defined on a small neighborhood suggests that we can only 
expect local convergence near the approximation point. It might be possible to overcome the problem that the finite-gap solutions are not defined for all $x$ by approximating the initial function not up to order $2N-1$, but only to a lower order and use the remaining freedom to find a solution that is complete. 
\newline

\textbf{Acknowledgements.} The participation of M.Q.\ at the MATRIX research program was supported through a MATRIX-Simons Young Scholar Award and the participation at workshop “Integrable 
Systems, Liouville Foliations and Nijenhuis Geometry Workshop” was supported by the Clay Mathematics Institute. W.S.\ thanks the UiT Aurora project MASCOT for their support. 
This work was initiated during the “Integrable 
Systems, Liouville Foliations and Nijenhuis Geometry Workshop” at 
CIMAT, Guanajuato (Mexico), 
and continued at the workshops “Integrable Systems and Differential 
Geometry” at La Trobe 
University, Melbourne (Australia) and “Nijenhuis Geometry, Haantjes 
Geometry and Separation of Variables” at MATRIX, Creswick 
(Australia), and we thank these institutions for their hospitality in 
hosting parts of this research. We thank Julia Bernatska, Alexey 
Bolsinov, Andrey Konyaev and Krzysztof Marciniak for useful 
discussions. We are particularly grateful to Boris 
Kruglikov and Vladimir S. Matveev for sketching the main approach in 
\cref{subsec:GeneralizationArbitraryPolynomial}.

\bibliographystyle{plainnat}
\bibliography{Bibliography}

\appendix

\section{Solutions of some recursion relations}

In this section, we give some useful recursion relations used in the paper.

\begin{lemma}[Recursions for the reduced tangent numbers] \label{lem:RecursionReducedTangentNumbers}
	The sequence of reduced tangent numbers $(t_i)_{i \geq 1}$ defined via
	\begin{align*}
		\tan(x) = \sum_{n=1}^{\infty} t_n \cdot 2^{n-1} \cdot \frac{x^{2n-1}}{(2n-1)!}
	\end{align*}
	satisfies $t_1=1$ and for $i>1$
	\begin{align*}
		t_i &= \frac{1}{2} \sum_{j=1}^{i-1} \binom{2i-2}{2j-1} t_j t_{i-j}.
	\end{align*}
\end{lemma}

\begin{proof}
	Apparently, $t_1=1$. Now we can use the formula $\tan'(x) = 1 + \tan(x)^2$ and apply this to the series:
	\begin{align*}
		\tan'(x) &= \sum_{n=1}^{\infty} t_n \cdot 2^{n-1} \cdot \frac{x^{2n-2}}{(2n-2)!}, \\
		1+\tan(x)^2 &= 1 + \sum_{n=2}^{\infty} \left(\sum_{k=1}^{n-1} t_k \cdot 2^{k-1} \cdot \frac{1}{(2k-2)!} \cdot t_{n-k} \cdot 2^{(n-k)-1} \cdot \frac{1}{(2(n-k)-2)!}\right) \cdot x^{2n-2} = \\ &=  1 + \sum_{n=2}^{\infty} \left(\frac{1}{2} \sum_{k=1}^{n-1} \binom{2n-2}{2k-1} t_k \cdot t_{n-k} \right) \cdot 2^{n-1}\frac{x^{2n-2}}{(2n-2)!}.
	\end{align*}
	A comparison of the coefficients shows that the \enquote{reduced tangent numbers} satisfy this recursion relation and the lemma is proven.
\end{proof}

\begin{lemma}[Solution of some recursion relation] \label{lem:solutionRecursionRelation} \quad \\
	A sequence $(s_i)_{i \geq 0}$ satisfying
	\begin{align*}
		s_i = - \sum_{j=0}^{i-1} \binom{2i-1}{2j} s_j \cdot (-1)^{i-j} t_{i-j}
	\end{align*}
	and $s_0=1$ is given by
	\begin{align*}
		s_i = \begin{cases}
			1 & i=0 \\
			2^{i-1}  & i>0 \\
		\end{cases}.
	\end{align*}
\end{lemma}

\begin{proof} \quad \\
	To solve this recursion, consider the power series
	\begin{align*}
		\sum_{n=1}^{\infty} t_n (-1)^{n} \cdot 2^{n-1} \cdot \frac{x^{2n-1}}{(2n-1)!} = \frac{\tan(\imath \cdot x)}{-\imath} = -\tanh(x).
	\end{align*}
	Now we make an ansatz
	\begin{align*}
		f(x) &= \sum_{n=0}^{\infty} s_n \cdot 2^{n} \cdot \frac{x^{2n}}{(2n)!}.
	\end{align*}
	Differentiating the power series, we obtain
	\begin{align*}
		f'(x) &= \sum_{n=1}^{\infty} s_n \cdot 2^{n} \cdot \frac{x^{2n-1}}{(2n-1)!}.
	\end{align*}
	Multiplying the power series for $-\tanh(x)$ and for $f(x)$ we obtain:
	\begin{align*}
		\sum_{n=0}^{\infty} s_n \cdot 2^{n} \cdot \frac{x^{2n}}{(2n)!} \cdot \sum_{n=1}^{\infty} t_n (-1)^{n} \cdot 2^{n-1} \cdot \frac{x^{2n-1}}{(2n-1)!} &= \sum_{n=1}^{\infty} \left(\sum_{k=0}^{n-1} s_k \frac{2^{k}}{(2k)!} \cdot  t_{n-k} (-1)^{n-k} \frac{2^{n-k-1}}{(2(n-k)-1)!} \right) x^{2n-1} \\
		&= \sum_{n=1}^{\infty} \frac{1}{2} \left(\sum_{k=0}^{n-1} \binom{2n-1}{2k} s_k \cdot  t_{n-k} (-1)^{n-k} \right)2^{n}\frac{ x^{2n-1}}{(2n-1)!} = \\
		&= \frac{1}{2} f'(x).
	\end{align*}
	So we obtain that the function $f$ satisfies the differential equation
	\begin{align*}
		f'(x) &= 2 f(x) \cdot \tanh(x).
	\end{align*}
	As $f(0) = 1$ follows from the initial values, we can find the solution of this ODE (e.g. with Maple, check it with the addition theorems for hyperbolic functions) as
	\begin{align*}
		f(x)&= \frac{\cosh(2x)+1}{2}.
	\end{align*}
	But we can calculate the power series for this function:
	\begin{align*}
		\sum_{n=0}^{\infty} s_n \cdot 2^{n} \cdot \frac{x^{2n}}{(2n)!} = f(x) &= \frac{\cosh(2x)+1}{2} = 1 + \frac{1}{2}\sum_{n=1}^{\infty} \frac{(2x)^{2n}}{(2n)!} =  1 + \sum_{n=1}^{\infty} 2^{n-1}2^{n}\frac{x^{2n}}{(2n)!}.
	\end{align*}
	So we obtain $s_n = 2^{n-1}$ for each $n>0$ as claimed.
\end{proof}

\begin{lemma}[Alternative recursion for the reduced tangent numbers] \label{lem:AlternativeRecursionReducedTangentNumbers} \quad \\
	For the reduced tangent numbers we have
	\begin{align*}
		t_{b+1} = \sum_{l=0}^{b-1} \binom{2b-1}{2l} t_{l+1} t_{b-l} = \sum_{l=1}^{b} \binom{2b-1}{2(l-1)} t_l t_{b+1-l}.
	\end{align*}
\end{lemma}

\begin{proof} \quad \\
	Shifting the index of $b$ by one, we need
	\begin{align*}
		t_{b} \stackrel{!}{=} \sum_{l=1}^{b-1} \binom{2(b-1)-1}{2(l-1)} t_l t_{b-l} = \sum_{l=1}^{b-1} \binom{2b-3}{2(l-1)} t_l t_{b-l}.
	\end{align*}
	Note that in Lemma~\ref{lem:RecursionReducedTangentNumbers} we already derived the recursion
	\begin{align*}
		t_b &= \frac{1}{2} \sum_{l=1}^{b-1} \binom{2b-2}{2l-1} t_l t_{b-l}.
	\end{align*}
	Now we consider the differences of the right-hand sides:
	\begin{align*}
		\frac{1}{2} \sum_{l=1}^{b-1} \binom{2b-2}{2l-1} t_l t_{b-l} - \sum_{l=1}^{b-1} \binom{2b-3}{2(l-1)} t_l t_{b-l} &= \sum_{l=1}^{b-1} \left(\frac{1}{2} \left(\binom{2b-3}{2l-1} + \binom{2b-3}{2l-2}\right) - \binom{2b-3}{2l-2} \right) t_l t_{b-l} = \\&= \frac{1}{2} \sum_{l=1}^{b-1} \left(\binom{2b-3}{2l-1} - \binom{2b-3}{2l-2}\right) t_l t_{b-l}.
	\end{align*}
	For $l\leq\frac{b-1}{2}$, there is another index $b-l\geq\frac{b+1}{2}$ that has the same factors $t_l$ and $t_{b-l}$. So we can combine the summands and obtain
	\begin{align*}
		\left(\binom{2b-3}{2l-1} - \binom{2b-3}{2l-2}\right) + \left(\binom{2b-3}{2(b-l)-1} - \binom{2b-3}{2(b-l)-2}\right) = 0,
	\end{align*}
	as $\binom{2b-3}{2l-1} = \binom{2b-3}{(2b-3)-(2l-1)} = \binom{2b-3}{(2b-l)-2}$ and analogously $\binom{2b-3}{2l-2} = \binom{2b-3}{2(b-l)-1}$. \\
	For $l=\frac{b}{2}$ we have
	\begin{align*}
		\binom{2b-3}{2l-1} - \binom{2b-3}{2l-2} &= \binom{2b-3}{b-1} - \binom{2b-3}{b-2} = 0,
	\end{align*}
	again by the properties of the binomial coefficient. \\
	So in total we get $\frac{1}{2} \sum_{l=1}^{b-1} \left(\binom{2b-3}{2l-1} - \binom{2b-3}{2l-2}\right) t_l t_{b-l} = 0$ and hence we proved the recursion for the reduced cotangent numbers. 
\end{proof}

\section{Illustrative Examples for $N = 6$} \label{sec:SomeMapleCalculations}

We depict some of the Jacobians for the case $N=6$ that we calculated. On the left we indicate the expression that we differentiate with respect to the variable in the upper row. Afterwards, we evaluate the resulting expressions at the point where the usual variables $w_1, \dots, w_{N-1}, p_1, \dots, p_N$ are set to $0$. The matrices contain many zeros and we use circles, dashed circles and dashed boxes for those which we can prove to be (non)zero for general $N$.

{\Tiny \allowdisplaybreaks
	\begin{align*}
		\left[\begin{array}{c|cccccccccccc}
			& \partial p_6 &  \partial w_1 & \partial p_5 & \partial w_2 & \partial p_4 & \partial w_3 & \partial p_3 & \partial w_4 & \partial p_2 & \partial w_5 & \partial p_1 & \partial w_6 \\ \hline
			(w_1)_x & \dashedbox{-1} & 0 & 0 & 0 & 0 & 0 & 0 & 0 & 0 & 0 & 0 & 0 
			\\
			(w_1)_{2x} & 0 & 2 & 0 & 0 & 0 & 0 & 0 & 0 & 0 & 0 & 0 & -\frac{2 c_{13}}{w_6^{3}} 
			\\
			(w_1)_{3x} & -2 & 0 & 0 & 0 & 0 & 0 & 0 & 0 & 0 & 0 & \frac{2 c_{13}}{w_6^{3}} & 0 
			\\
			(w_1)_{4x} &0 & 4 & 0 & 0 & 0 & 0 & 0 & 0 & 0 & -\frac{2 c_{13}^{2}}{w_6^{5}} & 0 & \frac{4 c_{13}}{w_6^{3}} 
			\\
			(w_1)_{5x} &-4 & 0 & 0 & 0 & 0 & 0 & 0 & 0 & \frac{4 c_{13}^{2}}{w_6^{5}} & 0 & -\frac{28 c_{13}}{w_6^{3}} & 0 
			\\
			(w_1)_{6x} &0 & 8 & 0 & 0 & 0 & 0 & 0 & -\frac{4 c_{13}^{3}}{w_6^{7}} & 0 & \frac{64 c_{13}^{2}}{w_6^{5}} & 0 & -\frac{120 c_{13}}{w_6^{3}} 
			\\
			(w_1)_{7x} &-8 & 0 & 0 & 0 & 0 & 0 & \frac{8 c_{13}^{3}}{w_6^{7}} & 0 & -\frac{248 c_{13}^{2}}{w_6^{5}} & 0 & \frac{1032 c_{13}}{w_6^{3}} & 0 
			\\
			(w_1)_{8x} &0 & 16 & 0 & 0 & 0 & -\frac{8 c_{13}^{4}}{w_6^{9}} & 0 & \frac{408 c_{13}^{3}}{w_6^{7}} & 0 & -\frac{3480 c_{13}^{2}}{w_6^{5}} & 0 & \frac{4112 c_{13}}{w_6^{3}} 
			\\
			(w_1)_{9x} &-16 & 0 & 0 & 0 & \frac{16 c_{13}^{4}}{w_6^{9}} & 0 & -\frac{1264 c_{13}^{3}}{w_6^{7}} & 0 & \frac{20496 c_{13}^{2}}{w_6^{5}} & 0 & -\frac{61424 c_{13}}{w_6^{3}} & 0 
			\\
			(w_1)_{10x} &0 & 32 & 0 & -\frac{16 c_{13}^{5}}{w_6^{11}} & 0 & \frac{1824 c_{13}^{4}}{w_6^{9}} & 0 & -\frac{47616 c_{13}^{3}}{w_6^{7}} & 0 & \frac{281312 c_{13}^{2}}{w_6^{5}} & 0 & -\frac{245728 c_{13}}{w_6^{3}} 
		\end{array}\right] \\
		\left[\begin{array}{c|cccccccccccc}
			& \partial p_6 &  \partial w_1 & \partial p_5 & \partial w_2 & \partial p_4 & \partial w_3 & \partial p_3 & \partial w_4 & \partial p_2 & \partial w_5 & \partial p_1 & \partial w_6 \\ \hline
			(w_2)_x & \circled{0} & \circled{0} & \dashedbox{-1} & 0 & 0 & 0 & 0 & 0 & 0 & 0 & 0 & 0 
			\\
			(w_2)_{2x} & \dashedcircle{0} & \dashedcircle{\frac{c_{13}}{w_6^{2}}} & 0 & 2 & 0 & 0 & 0 & 0 & 0 & 0 & 0 & 0 
			\\
			(w_2)_{3x} & \dashedbox{-\frac{2 c_{13}}{w_6^{2}}} & 0 & -2 & 0 & 0 & 0 & 0 & 0 & 0 & 0 & 0 & 0 
			\\
			(w_2)_{4x} & 0 & \frac{2 c_{13}}{w_6^{2}} & 0 & 4 & 0 & 0 & 0 & 0 & 0 & 0 & 0 & -\frac{8 c_{13}^{2}}{w_6^{5}} 
			\\
			(w_2)_{5x} & \frac{8 c_{13}}{w_6^{2}} & 0 & -4 & 0 & 0 & 0 & 0 & 0 & 0 & 0 & \frac{14 c_{13}^{2}}{w_6^{5}} & 0 
			\\
			(w_2)_{6x} & 0 & \frac{4 c_{13}}{w_6^{2}} & 0 & 8 & 0 & 0 & 0 & 0 & 0 & -\frac{22 c_{13}^{3}}{w_6^{7}} & 0 & \frac{144 c_{13}^{2}}{w_6^{5}} 
			\\
			(w_2)_{7x} & -\frac{248 c_{13}}{w_6^{2}} & 0 & -8 & 0 & 0 & 0 & 0 & 0 & \frac{64 c_{13}^{3}}{w_6^{7}} & 0 & -\frac{640 c_{13}^{2}}{w_6^{5}} & 0 
			\\
			(w_2)_{8x} & 0 & \frac{8 c_{13}}{w_6^{2}} & 0 & 16 & 0 & 0 & 0 & -\frac{88 c_{13}^{4}}{w_6^{9}} & 0 & \frac{1744 c_{13}^{3}}{w_6^{7}} & 0 & -\frac{6112 c_{13}^{2}}{w_6^{5}} 
			\\
			(w_2)_{9x} &\frac{11072 c_{13}}{w_6^{2}} & 0 & -16 & 0 & 0 & 0 & \frac{232 c_{13}^{4}}{w_6^{9}} & 0 & -\frac{8296 c_{13}^{3}}{w_6^{7}} & 0 & \frac{45672 c_{13}^{2}}{w_6^{5}} & 0 
			\\
			(w_2)_{10x} & 0 & \frac{16 c_{13}}{w_6^{2}} & 0 & 32 & 0 & -\frac{296 c_{13}^{5}}{w_6^{11}} & 0 & \frac{16824 c_{13}^{4}}{w_6^{9}} & 0 & -\frac{175896 c_{13}^{3}}{w_6^{7}} & 0 & \frac{409664 c_{13}^{2}}{w_6^{5}} 
		\end{array}\right]
		\\
		\left[\begin{array}{c|cccccccccccc}
			& \partial p_6 &  \partial w_1 & \partial p_5 & \partial w_2 & \partial p_4 & \partial w_3 & \partial p_3 & \partial w_4 & \partial p_2 & \partial w_5 & \partial p_1 & \partial w_6 \\ \hline
			(w_3)_{x} & \circled{0} & \circled{0} & \circled{0} & \circled{0} & \dashedbox{-1} & 0 & 0 & 0 & 0 & 0 & 0 & 0 
			\\
			(w_3)_{2x} & \circled{0} & \circled{0} & \dashedcircle{0} & \dashedcircle{\frac{c_{13}}{w_6^{2}}} & 0 & 2 & 0 & 0 & 0 & 0 & 0 & 0 
			\\
			(w_3)_{3x} & \circled{0} & \circled{0} & \dashedbox{-\frac{2 c_{13}}{w_6^{2}}} & 0 & -2 & 0 & 0 & 0 & 0 & 0 & 0 & 0 
			\\
			(w_3)_{4x} & \dashedcircle{0} & \dashedcircle{\frac{2 c_{13}^{2}}{w_6^{4}}} & 0 & \frac{2 c_{13}}{w_6^{2}} & 0 & 4 & 0 & 0 & 0 & 0 & 0 & 0 
			\\
			(w_3)_{5x} & \dashedbox{-\frac{4 c_{13}^{2}}{w_6^{4}}} & 0 & \frac{8 c_{13}}{w_6^{2}} & 0 & -4 & 0 & 0 & 0 & 0 & 0 & 0 & 0 
			\\
			(w_3)_{6x} & 0 & -\frac{28 c_{13}^{2}}{w_6^{4}} & 0 & \frac{4 c_{13}}{w_6^{2}} & 0 & 8 & 0 & 0 & 0 & 0 & 0 & -\frac{24 c_{13}^{3}}{w_6^{7}} 
			\\
			(w_3)_{7x} & \frac{136 c_{13}^{2}}{w_6^{4}} & 0 & -\frac{248 c_{13}}{w_6^{2}} & 0 & -8 & 0 & 0 & 0 & 0 & 0 & \frac{64 c_{13}^{3}}{w_6^{7}} & 0 
			\\
			(w_3)_{8x} & 0 & \frac{1032 c_{13}^{2}}{w_6^{4}} & 0 & \frac{8 c_{13}}{w_6^{2}} & 0 & 16 & 0 & 0 & 0 & -\frac{144 c_{13}^{4}}{w_6^{9}} & 0 & \frac{1584 c_{13}^{3}}{w_6^{7}} 
			\\
			(w_3)_{9x} & -\frac{7776 c_{13}^{2}}{w_6^{4}} & 0 & \frac{11072 c_{13}}{w_6^{2}} & 0 & -16 & 0 & 0 & 0 & \frac{568 c_{13}^{4}}{w_6^{9}} & 0 & -\frac{7624 c_{13}^{3}}{w_6^{7}} & 0 
			\\
			(w_3)_{10x} & 0 & -\frac{61424 c_{13}^{2}}{w_6^{4}} & 0 & \frac{16 c_{13}}{w_6^{2}} & 0 & 32 & 0 & -\frac{1016 c_{13}^{5}}{w_6^{11}} & 0 & \frac{25320 c_{13}^{4}}{w_6^{9}} & 0 & -\frac{138144 c_{13}^{3}}{w_6^{7}} 
		\end{array}\right] \\
		\left[\begin{array}{c|cccccccccccc}
			& \partial p_6 &  \partial w_1 & \partial p_5 & \partial w_2 & \partial p_4 & \partial w_3 & \partial p_3 & \partial w_4 & \partial p_2 & \partial w_5 & \partial p_1 & \partial w_6 \\ \hline
			(w_4)_{x} &	\circled{0} & \circled{0} & \circled{0} & \circled{0} & \circled{0} & \circled{0} & \dashedbox{-1} & 0 & 0 & 0 & 0 & 0 
			\\
			(w_4)_{2x} &	\circled{0} & \circled{0} & \circled{0} & \circled{0} & \dashedcircle{0} & \dashedcircle{\frac{c_{13}}{w_6^{2}}} & 0 & 2 & 0 & 0 & 0 & 0 
			\\
			(w_4)_{3x}	& \circled{0} & \circled{0} & \circled{0} & \circled{0} & \dashedbox{-\frac{2 c_{13}}{w_6^{2}}} & 0 & -2 & 0 & 0 & 0 & 0 & 0 
			\\
			(w_4)_{4x} &	\circled{0} & \circled{0} & \dashedcircle{0} & \dashedcircle{\frac{2 c_{13}^{2}}{w_6^{4}}} & 0 & \frac{2 c_{13}}{w_6^{2}} & 0 & 4 & 0 & 0 & 0 & 0 
			\\
			(w_4)_{5x} &  \circled{0} & \circled{0} & \dashedbox{-\frac{4 c_{13}^{2}}{w_6^{4}}} & 0 & \frac{8 c_{13}}{w_6^{2}} & 0 & -4 & 0 & 0 & 0 & 0 & 0 
			\\
			(w_4)_{6x} & \dashedcircle{0} & \dashedcircle{\frac{4 c_{13}^{3}}{w_6^{6}}} & 0 & -\frac{28 c_{13}^{2}}{w_6^{4}} & 0 & \frac{4 c_{13}}{w_6^{2}} & 0 & 8 & 0 & 0 & 0 & 0 
			\\
			(w_4)_{7x} & \dashedbox{-\frac{8 c_{13}^{3}}{w_6^{6}}} & 0 & \frac{136 c_{13}^{2}}{w_6^{4}} & 0 & -\frac{248 c_{13}}{w_6^{2}} & 0 & -8 & 0 & 0 & 0 & 0 & 0 
			\\
			(w_4)_{8x} & 0 & -\frac{248 c_{13}^{3}}{w_6^{6}} & 0 & \frac{1032 c_{13}^{2}}{w_6^{4}} & 0 & \frac{8 c_{13}}{w_6^{2}} & 0 & 16 & 0 & 0 & 0 & -\frac{64 c_{13}^{4}}{w_6^{9}} 
			\\
			(w_4)_{9x} & \frac{832 c_{13}^{3}}{w_6^{6}} & 0 & -\frac{7776 c_{13}^{2}}{w_6^{4}} & 0 & \frac{11072 c_{13}}{w_6^{2}} & 0 & -16 & 0 & 0 & 0 & \frac{232 c_{13}^{4}}{w_6^{9}} & 0 
			\\
			(w_4)_{10x} & 0 & \frac{20496 c_{13}^{3}}{w_6^{6}} & 0 & -\frac{61424 c_{13}^{2}}{w_6^{4}} & 0 & \frac{16 c_{13}}{w_6^{2}} & 0 & 32 & 0 & -\frac{680 c_{13}^{5}}{w_6^{11}} & 0 & \frac{10368 c_{13}^{4}}{w_6^{9}} 
		\end{array}\right] \\
		\left[\begin{array}{c|cccccccccccc}
			& \partial p_6 &  \partial w_1 & \partial p_5 & \partial w_2 & \partial p_4 & \partial w_3 & \partial p_3 & \partial w_4 & \partial p_2 & \partial w_5 & \partial p_1 & \partial w_6 \\ \hline
			(w_5)_{x} &	\circled{0} & \circled{0} & \circled{0} & \circled{0} & \circled{0} & \circled{0} & \circled{0} & \circled{0} & \dashedbox{-1} & 0 & 0 & 0 
			\\
			(w_5)_{2x} &	\circled{0} & \circled{0} & \circled{0} & \circled{0} & \circled{0} & \circled{0} & \dashedcircle{0} & \dashedcircle{\frac{c_{13}}{w_6^{2}}} & 0 & 2 & 0 & 0 
			\\
			(w_5)_{3x} &\circled{0} & \circled{0} & \circled{0} & \circled{0} & \circled{0} & \circled{0} & \dashedbox{-\frac{2 c_{13}}{w_6^{2}}} & 0 & -2 & 0 & 0 & 0 
			\\
			(w_5)_{4x} & \circled{0} & \circled{0} & \circled{0} & \circled{0} & \dashedcircle{0} & \dashedcircle{\frac{2 c_{13}^{2}}{w_6^{4}}} & 0 & \frac{2 c_{13}}{w_6^{2}} & 0 & 4 & 0 & 0 
			\\
			(w_5)_{5x} &	\circled{0} & \circled{0} & \circled{0} & \circled{0} & \dashedbox{-\frac{4 c_{13}^{2}}{w_6^{4}}} & 0 & \frac{8 c_{13}}{w_6^{2}} & 0 & -4 & 0 & 0 & 0 
			\\
			(w_5)_{6x} &	\circled{0} & \circled{0} & \dashedcircle{0} & \dashedcircle{\frac{4 c_{13}^{3}}{w_6^{6}}} & 0 & -\frac{28 c_{13}^{2}}{w_6^{4}} & 0 & \frac{4 c_{13}}{w_6^{2}} & 0 & 8 & 0 & 0 
			\\
			(w_5)_{7x} &	\circled{0} & \circled{0} & \dashedbox{-\frac{8 c_{13}^{3}}{w_6^{6}}} & 0 & \frac{136 c_{13}^{2}}{w_6^{4}} & 0 & -\frac{248 c_{13}}{w_6^{2}} & 0 & -8 & 0 & 0 & 0 
			\\
			(w_5)_{8x} &	\dashedcircle{0} & \dashedcircle{\frac{8 c_{13}^{4}}{w_6^{8}}} & 0 & -\frac{248 c_{13}^{3}}{w_6^{6}} & 0 & \frac{1032 c_{13}^{2}}{w_6^{4}} & 0 & \frac{8 c_{13}}{w_6^{2}} & 0 & 16 & 0 & 0 
			\\
			(w_5)_{9x} &	\dashedbox{-\frac{16 c_{13}^{4}}{w_6^{8}}} & 0 & \frac{832 c_{13}^{3}}{w_6^{6}} & 0 & -\frac{7776 c_{13}^{2}}{w_6^{4}} & 0 & \frac{11072 c_{13}}{w_6^{2}} & 0 & -16 & 0 & 0 & 0 
			\\
			(w_5)_{10x} &	0 & -\frac{1264 c_{13}^{4}}{w_6^{8}} & 0 & \frac{20496 c_{13}^{3}}{w_6^{6}} & 0 & -\frac{61424 c_{13}^{2}}{w_6^{4}} & 0 & \frac{16 c_{13}}{w_6^{2}} & 0 & 32 & 0 & -\frac{160 c_{13}^{5}}{w_6^{11}} 
		\end{array}\right] \\
	\end{align*}
}
Next for the momenta, again in the order of increasing degree:
{
	\fontsize{4pt}{5pt}\selectfont
	\allowdisplaybreaks
	\begin{align*}
		\left[\begin{array}{c|cccccccccccc}
			& \partial p_6 &  \partial w_1 & \partial p_5 & \partial w_2 & \partial p_4 & \partial w_3 & \partial p_3 & \partial w_4 & \partial p_2 & \partial w_5 & \partial p_1 & \partial w_6 \\ \hline
			(p_6)_{x} &	0 & -2 & 0 & 0 & 0 & 0 & 0 & 0 & 0 & 0 & 0 & \frac{2 c_{13}}{w_6^{3}} 
			\\
			(p_6)_{2x} &	2 & 0 & 0 & 0 & 0 & 0 & 0 & 0 & 0 & 0 & -\frac{2 c_{13}}{w_6^{3}} & 0 
			\\
			(p_6)_{3x} &	0 & -4 & 0 & 0 & 0 & 0 & 0 & 0 & 0 & \frac{2 c_{13}^{2}}{w_6^{5}} & 0 & -\frac{4 c_{13}}{w_6^{3}} 
			\\
			(p_6)_{4x} &	4 & 0 & 0 & 0 & 0 & 0 & 0 & 0 & -\frac{4 c_{13}^{2}}{w_6^{5}} & 0 & \frac{28 c_{13}}{w_6^{3}} & 0 
			\\
			(p_6)_{5x} & 0 & -8 & 0 & 0 & 0 & 0 & 0 & \frac{4 c_{13}^{3}}{w_6^{7}} & 0 & -\frac{64 c_{13}^{2}}{w_6^{5}} & 0 & \frac{120 c_{13}}{w_6^{3}} 
			\\
			(p_6)_{6x} &8 & 0 & 0 & 0 & 0 & 0 & -\frac{8 c_{13}^{3}}{w_6^{7}} & 0 & \frac{248 c_{13}^{2}}{w_6^{5}} & 0 & -\frac{1032 c_{13}}{w_6^{3}} & 0 
			\\
			(p_6)_{7x} &0 & -16 & 0 & 0 & 0 & \frac{8 c_{13}^{4}}{w_6^{9}} & 0 & -\frac{408 c_{13}^{3}}{w_6^{7}} & 0 & \frac{3480 c_{13}^{2}}{w_6^{5}} & 0 & -\frac{4112 c_{13}}{w_6^{3}} 
			\\
			(p_6)_{8x} &16 & 0 & 0 & 0 & -\frac{16 c_{13}^{4}}{w_6^{9}} & 0 & \frac{1264 c_{13}^{3}}{w_6^{7}} & 0 & -\frac{20496 c_{13}^{2}}{w_6^{5}} & 0 & \frac{61424 c_{13}}{w_6^{3}} & 0 
			\\
			(p_6)_{9x} &0 & -32 & 0 & \frac{16 c_{13}^{5}}{w_6^{11}} & 0 & -\frac{1824 c_{13}^{4}}{w_6^{9}} & 0 & \frac{47616 c_{13}^{3}}{w_6^{7}} & 0 & -\frac{281312 c_{13}^{2}}{w_6^{5}} & 0 & \frac{245728 c_{13}}{w_6^{3}} 
			\\
			(p_6)_{10x} &32 & 0 & -\frac{32 c_{13}^{5}}{w_6^{11}} & 0 & \frac{5088 c_{13}^{4}}{w_6^{9}} & 0 & -\frac{204832 c_{13}^{3}}{w_6^{7}} & 0 & \frac{2297824 c_{13}^{2}}{w_6^{5}} & 0 & -\frac{5537824 c_{13}}{w_6^{3}} & 0 
		\end{array}\right] \\
		\left[\begin{array}{c|cccccccccccc}
			& \partial p_6 &  \partial w_1 & \partial p_5 & \partial w_2 & \partial p_4 & \partial w_3 & \partial p_3 & \partial w_4 & \partial p_2 & \partial w_5 & \partial p_1 & \partial w_6 \\ \hline
			(p_5)_{x} & \circled{0} & \circled{0} & 0 & -2 & 0 & 0 & 0 & 0 & 0 & 0 & 0 & 0 
			\\
			(p_5)_{2x} &	\dashedcircle{-\frac{c_{13}}{w_6^{2}}} & \dashedcircle{0} & 2 & 0 & 0 & 0 & 0 & 0 & 0 & 0 & 0 & 0 
			\\
			(p_5)_{3x} &	0 & \frac{8 c_{13}}{w_6^{2}} & 0 & -4 & 0 & 0 & 0 & 0 & 0 & 0 & 0 & -\frac{4 c_{13}^{2}}{w_6^{5}} 
			\\
			(p_5)_{4x} &	-\frac{18 c_{13}}{w_6^{2}} & 0 & 4 & 0 & 0 & 0 & 0 & 0 & 0 & 0 & \frac{6 c_{13}^{2}}{w_6^{5}} & 0 
			\\
			(p_5)_{5x} &	0 & \frac{56 c_{13}}{w_6^{2}} & 0 & -8 & 0 & 0 & 0 & 0 & 0 & -\frac{8 c_{13}^{3}}{w_6^{7}} & 0 & -\frac{24 c_{13}^{2}}{w_6^{5}} 
			\\
			(p_5)_{6x} & -\frac{116 c_{13}}{w_6^{2}} & 0 & 8 & 0 & 0 & 0 & 0 & 0 & \frac{20 c_{13}^{3}}{w_6^{7}} & 0 & -\frac{88 c_{13}^{2}}{w_6^{5}} & 0 
			\\
			(p_5)_{7x} &0 & \frac{960 c_{13}}{w_6^{2}} & 0 & -16 & 0 & 0 & 0 & -\frac{24 c_{13}^{4}}{w_6^{9}} & 0 & \frac{328 c_{13}^{3}}{w_6^{7}} & 0 & -\frac{1168 c_{13}^{2}}{w_6^{5}} 
			\\
			(p_5)_{8x} &-\frac{1928 c_{13}}{w_6^{2}} & 0 & 16 & 0 & 0 & 0 & \frac{56 c_{13}^{4}}{w_6^{9}} & 0 & -\frac{1640 c_{13}^{3}}{w_6^{7}} & 0 & \frac{8616 c_{13}^{2}}{w_6^{5}} & 0 
			\\
			(p_5)_{9x} & 0 & -\frac{14432 c_{13}}{w_6^{2}} & 0 & -32 & 0 & -\frac{64 c_{13}^{5}}{w_6^{11}} & 0 & \frac{3216 c_{13}^{4}}{w_6^{9}} & 0 & -\frac{32784 c_{13}^{3}}{w_6^{7}} & 0 & \frac{61216 c_{13}^{2}}{w_6^{5}} 
			\\
			(p_5)_{10x} &\frac{28848 c_{13}}{w_6^{2}} & 0 & 32 & 0 & \frac{144 c_{13}^{5}}{w_6^{11}} & 0 & -\frac{11456 c_{13}^{4}}{w_6^{9}} & 0 & \frac{213344 c_{13}^{3}}{w_6^{7}} & 0 & -\frac{851520 c_{13}^{2}}{w_6^{5}} & 0 
		\end{array}\right] \\
		\left[\begin{array}{c|cccccccccccc}
			& \partial p_6 &  \partial w_1 & \partial p_5 & \partial w_2 & \partial p_4 & \partial w_3 & \partial p_3 & \partial w_4 & \partial p_2 & \partial w_5 & \partial p_1 & \partial w_6 \\ \hline
			(p_4)_{x} &\circled{0} & \circled{0} & \circled{0} & \circled{0} & 0 & -2 & 0 & 0 & 0 & 0 & 0 & 0 
			\\
			(p_4)_{2x} &	\circled{0} & \circled{0} & \dashedcircle{-\frac{c_{13}}{w_6^{2}}} & \dashedcircle{0} & 2 & 0 & 0 & 0 & 0 & 0 & 0 & 0 
			\\
			(p_4)_{3x} &	\circled{0} & \circled{0} & 0 & \frac{8 c_{13}}{w_6^{2}} & 0 & -4 & 0 & 0 & 0 & 0 & 0 & 0 
			\\
			(p_4)_{4x} &	\dashedcircle{\frac{4 c_{13}^{2}}{w_6^{4}}} & \dashedcircle{0} & -\frac{18 c_{13}}{w_6^{2}} & 0 & 4 & 0 & 0 & 0 & 0 & 0 & 0 & 0 
			\\
			(p_4)_{5x} &	0 & -\frac{68 c_{13}^{2}}{w_6^{4}} & 0 & \frac{56 c_{13}}{w_6^{2}} & 0 & -8 & 0 & 0 & 0 & 0 & 0 & \frac{24 c_{13}^{3}}{w_6^{7}} 
			\\
			(p_4)_{6x} &	\frac{284 c_{13}^{2}}{w_6^{4}} & 0 & -\frac{116 c_{13}}{w_6^{2}} & 0 & 8 & 0 & 0 & 0 & 0 & 0 & -\frac{50 c_{13}^{3}}{w_6^{7}} & 0 
			\\
			(p_4)_{7x} &	0 & -\frac{1408 c_{13}^{2}}{w_6^{4}} & 0 & \frac{960 c_{13}}{w_6^{2}} & 0 & -16 & 0 & 0 & 0 & \frac{88 c_{13}^{4}}{w_6^{9}} & 0 & \frac{1104 c_{13}^{3}}{w_6^{7}} 
			\\
			(p_4)_{8x} &	\frac{3744 c_{13}^{2}}{w_6^{4}} & 0 & -\frac{1928 c_{13}}{w_6^{2}} & 0 & 16 & 0 & 0 & 0 & -\frac{280 c_{13}^{4}}{w_6^{9}} & 0 & -\frac{800 c_{13}^{3}}{w_6^{7}} & 0 
			\\
			(p_4)_{9x} &	0 & -\frac{50832 c_{13}^{2}}{w_6^{4}} & 0 & -\frac{14432 c_{13}}{w_6^{2}} & 0 & -32 & 0 & \frac{416 c_{13}^{5}}{w_6^{11}} & 0 & -\frac{1440 c_{13}^{4}}{w_6^{9}} & 0 & \frac{12864 c_{13}^{3}}{w_6^{7}} 
			\\
			(p_4)_{10x} &	\frac{214928 c_{13}^{2}}{w_6^{4}} & 0 & \frac{28848 c_{13}}{w_6^{2}} & 0 & 32 & 0 & -\frac{1176 c_{13}^{5}}{w_6^{11}} & 0 & \frac{18376 c_{13}^{4}}{w_6^{9}} & 0 & -\frac{72040 c_{13}^{3}}{w_6^{7}} & 0 
		\end{array}\right] \\
		\left[\begin{array}{c|cccccccccccc}
			& \partial p_6 &  \partial w_1 & \partial p_5 & \partial w_2 & \partial p_4 & \partial w_3 & \partial p_3 & \partial w_4 & \partial p_2 & \partial w_5 & \partial p_1 & \partial w_6 \\ \hline
			(p_3)_{x} &\circled{0} & \circled{0} & \circled{0} & \circled{0} & \circled{0} & \circled{0} & 0 & -2 & 0 & 0 & 0 & 0 
			\\
			(p_3)_{2x} &	\circled{0} & \circled{0} & \circled{0} & \circled{0} & \dashedcircle{-\frac{c_{13}}{w_6^{2}}} & \dashedcircle{0} & 2 & 0 & 0 & 0 & 0 & 0 
			\\
			(p_3)_{3x} &	\circled{0} & \circled{0} & \circled{0} & \circled{0} & 0 & \frac{8 c_{13}}{w_6^{2}} & 0 & -4 & 0 & 0 & 0 & 0 
			\\
			(p_3)_{4x} &	\circled{0} & \circled{0} & \dashedcircle{\frac{4 c_{13}^{2}}{w_6^{4}}} & \dashedcircle{0} & -\frac{18 c_{13}}{w_6^{2}} & 0 & 4 & 0 & 0 & 0 & 0 & 0 
			\\
			(p_3)_{5x} &	\circled{0} & \circled{0} & 0 & -\frac{68 c_{13}^{2}}{w_6^{4}} & 0 & \frac{56 c_{13}}{w_6^{2}} & 0 & -8 & 0 & 0 & 0 & 0 
			\\
			(p_3)_{6x} &	\dashedcircle{-\frac{34 c_{13}^{3}}{w_6^{6}}} & \dashedcircle{0} & \frac{284 c_{13}^{2}}{w_6^{4}} & 0 & -\frac{116 c_{13}}{w_6^{2}} & 0 & 8 & 0 & 0 & 0 & 0 & 0 
			\\
			(p_3)_{7x} &	0 & \frac{992 c_{13}^{3}}{w_6^{6}} & 0 & -\frac{1408 c_{13}^{2}}{w_6^{4}} & 0 & \frac{960 c_{13}}{w_6^{2}} & 0 & -16 & 0 & 0 & 0 & -\frac{272 c_{13}^{4}}{w_6^{9}} 
			\\
			(p_3)_{8x} &	-\frac{6832 c_{13}^{3}}{w_6^{6}} & 0 & \frac{3744 c_{13}^{2}}{w_6^{4}} & 0 & -\frac{1928 c_{13}}{w_6^{2}} & 0 & 16 & 0 & 0 & 0 & \frac{728 c_{13}^{4}}{w_6^{9}} & 0 
			\\
			(p_3)_{9x} &	0 & \frac{53888 c_{13}^{3}}{w_6^{6}} & 0 & -\frac{50832 c_{13}^{2}}{w_6^{4}} & 0 & -\frac{14432 c_{13}}{w_6^{2}} & 0 & -32 & 0 & -\frac{1600 c_{13}^{5}}{w_6^{11}} & 0 & -\frac{43008 c_{13}^{4}}{w_6^{9}} 
			\\
			(p_3)_{10x} &	-\frac{174712 c_{13}^{3}}{w_6^{6}} & 0 & \frac{214928 c_{13}^{2}}{w_6^{4}} & 0 & \frac{28848 c_{13}}{w_6^{2}} & 0 & 32 & 0 & \frac{6216 c_{13}^{5}}{w_6^{11}} & 0 & \frac{88072 c_{13}^{4}}{w_6^{9}} & 0 
		\end{array}\right] \\
		\left[\begin{array}{c|cccccccccccc}
			& \partial p_6 &  \partial w_1 & \partial p_5 & \partial w_2 & \partial p_4 & \partial w_3 & \partial p_3 & \partial w_4 & \partial p_2 & \partial w_5 & \partial p_1 & \partial w_6 \\ \hline
			(p_2)_{x} &\circled{0} & \circled{0} & \circled{0} & \circled{0} & \circled{0} & \circled{0} & \circled{0} & \circled{0} & -2 & 0 & 0 
			\\
			(p_2)_{2x} &\circled{0} & \circled{0} & \circled{0} & \circled{0} & \circled{0} & \circled{0} & \dashedcircle{-\frac{c_{13}}{w_6^{2}}} & \dashedcircle{0} & 2 & 0 & 0 & 0 
			\\
			(p_2)_{3x} &\circled{0} & \circled{0} & \circled{0} & \circled{0} & \circled{0} & \circled{0} & 0 & \frac{8 c_{13}}{w_6^{2}} & 0 & -4 & 0 & 0 
			\\
			(p_2)_{4x} &\circled{0} & \circled{0} & \circled{0} & \circled{0} & \dashedcircle{\frac{4 c_{13}^{2}}{w_6^{4}}} & \dashedcircle{0} & -\frac{18 c_{13}}{w_6^{2}} & 0 & 4 & 0 & 0 & 0 
			\\
			(p_2)_{5x} &\circled{0} & \circled{0} & \circled{0} & \circled{0} & 0 & -\frac{68 c_{13}^{2}}{w_6^{4}} & 0 & \frac{56 c_{13}}{w_6^{2}} & 0 & -8 & 0 & 0 
			\\
			(p_2)_{6x} &\circled{0} & \circled{0} & \dashedcircle{-\frac{34 c_{13}^{3}}{w_6^{6}}} & \dashedcircle{0} & \frac{284 c_{13}^{2}}{w_6^{4}} & 0 & -\frac{116 c_{13}}{w_6^{2}} & 0 & 8 & 0 & 0 & 0 
			\\
			(p_2)_{7x} &\circled{0} & \circled{0} & 0 & \frac{992 c_{13}^{3}}{w_6^{6}} & 0 & -\frac{1408 c_{13}^{2}}{w_6^{4}} & 0 & \frac{960 c_{13}}{w_6^{2}} & 0 & -16 & 0 & 0 
			\\
			(p_2)_{8x} &\dashedcircle{\frac{496 c_{13}^{4}}{w_6^{8}}} & \dashedcircle{0} & -\frac{6832 c_{13}^{3}}{w_6^{6}} & 0 & \frac{3744 c_{13}^{2}}{w_6^{4}} & 0 & -\frac{1928 c_{13}}{w_6^{2}} & 0 & 16 & 0 & 0 & 0 
			\\
			(p_2)_{9x} &0 & -\frac{22112 c_{13}^{4}}{w_6^{8}} & 0 & \frac{53888 c_{13}^{3}}{w_6^{6}} & 0 & -\frac{50832 c_{13}^{2}}{w_6^{4}} & 0 & -\frac{14432 c_{13}}{w_6^{2}} & 0 & -32 & 0 & \frac{4960 c_{13}^{5}}{w_6^{11}} 
			\\
			(p_2)_{10x} &\frac{230048 c_{13}^{4}}{w_6^{8}} & 0 & -\frac{174712 c_{13}^{3}}{w_6^{6}} & 0 & \frac{214928 c_{13}^{2}}{w_6^{4}} & 0 & \frac{28848 c_{13}}{w_6^{2}} & 0 & 32 & 0 & -\frac{16224 c_{13}^{5}}{w_6^{11}} & 0 
		\end{array}\right]
	\end{align*}
}

\end{document}